\documentclass[11pt]{article}
\usepackage[utf8]{inputenc}
\usepackage{manuscript}

\DeclareAcronym{SRN}{
  short=SRN,
  long=Stochastic Reaction Network,
}

\DeclareAcronym{FSP}{
  short=FSP,
  long=Finite State Projection,
}

\DeclareAcronym{MC}{
  short=MC,
  long=Monte Carlo,
}

\DeclareAcronym{SSA}{
  short=SSA,
  long=Stochastic Simulation Algorithm,
}

\DeclareAcronym{TL}{
  short=TL,
  long=Tau Leap,
}

\DeclareAcronym{FFSP}{
  short=FFSP,
  long=Filtered Finite State Projection,
}

\DeclareAcronym{PF}{
  short=PF,
  long=Particle Filter,
}

\DeclareAcronym{CME}{
  short=CME,
  long=Chemical Master Equation,
}

\DeclareAcronym{ODE}{
  short=ODE,
  long=Ordinary Differential Equation,
}

\DeclareAcronym{FMP}{
  short=FMP,
  long=Filtered Markovian Projection,
}

\DeclareAcronym{MP}{
  short=MP,
  long=Markovian Projection,
}

\DeclareAcronym{PMF}{
  short=PMF,
  long=Probability Mass Function,
}

\DeclareAcronym{QOI}{
  short=QoI,
  long=Quantity of Interest,
}

\DeclareAcronym{CMP}{
  short=CMP,
  long=Conditional Markovian Projection,
}

\DeclareAcronym{UMP}{
  short=UMP,
  long=Unconditional Markovian Projection,
}

\title{Filtered Markovian Projection: Dimensionality Reduction in Filtering for Stochastic Reaction Networks}

\author[1]{Chiheb Ben Hammouda}
\author[2]{Maksim Chupin\thanks{\href{mailto:maksim.chupin@kaust.edu.sa}{maksim.chupin@kaust.edu.sa}}}
\author[3]{Sophia M\"{u}nker}
\author[2,3]{Ra\'{u}l Tempone}

\affil[1]{\small{Mathematical Institute, Utrecht University, Utrecht, the Netherlands}}
\affil[2]{\small{King Abdullah University of Science and Technology (KAUST), Computer, Electrical and Mathematical Sciences \& Engineering Division (CEMSE), Thuwal, Saudi Arabia}}
\affil[3]{\small{Chair of Mathematics for Uncertainty Quantification, RWTH Aachen University, Aachen, Germany}}

\begin{document}
    
    \maketitle
    
    \begin{abstract}
        Stochastic reaction networks (SRNs) model stochastic effects for various applications, including intracellular chemical or biological processes and epidemiology. A typical challenge in practical problems modeled by SRNs is that only a few state variables can be dynamically observed. Given the measurement trajectories, one can estimate the conditional probability distribution of unobserved (hidden) state variables by solving a stochastic filtering problem. In this setting, the conditional distribution evolves over time according to an extensive or potentially infinite-dimensional system of coupled ordinary differential equations with jumps, known as the filtering equation. The current numerical filtering techniques, such as the filtered finite state projection \cite{DAmbrosio2022FFSP}, are hindered by the curse of dimensionality, significantly affecting their computational performance. To address these limitations, we propose to use a dimensionality reduction technique based on the Markovian projection (MP), initially introduced for forward problems \cite{Hammouda2023MP}. In this work, we explore how to adapt the existing MP approach to the filtering problem and introduce a novel version of the MP, the Filtered MP, that guarantees the consistency of the resulting estimator. The novel consistent MP filter employs a reduced-variance particle filter for estimating the jump intensities of the projected model and solves the filtering equations in a low-dimensional space. The analysis and empirical results highlight the superior computational efficiency of projection methods compared to the existing filtered finite state projection in the large dimensional setting.
    \end{abstract}
    
    \sloppy
    
    \fussy
    
    \section{Introduction}
        \label{sec:0_Intro}
        This paper provides a framework for dimensionality reduction in a filtering problem for a special class of continuous-time Markov chains, \acfp{SRN}. This work considers a partially observed \acp{SRN} with a high-dimensional hidden state space and exact (noise-free and continuous in time) observations. The goal is to estimate the conditional distribution of the hidden states for a given observed trajectory. We employ the \acf{MP} technique originally introduced for dimensionality reduction in forward problems \cite{Hammouda2023MP} and extend it to the filtering setup. The aim is to determine an \ac{SRN} of lower dimensionality that preserves the marginal conditional distributions of the original system. Based on the \acf{FMP} theorem, we present a novel filtering algorithm. The novel approach, called \acf{MP} filter, integrates the strengths of the \ac{MP}, \acf{PF} \cite{Rathinam2021PFwithExactState}, and \acf{FFSP} \cite{DAmbrosio2022FFSP}. It comprises two key steps: estimating the propensities of the  projected \ac{SRN} using a computationally efficient \ac{PF} with a sufficiently small number of particles and  solving the resulting low-dimensional filtering problem using \ac{FFSP}.

The foundations of nonlinear filtering theory were established by several researchers \cite{Stratonovich1965Conditional, Kushner1967Dynamical, Zakai1969Optimal}. Since then, various numerical algorithms for estimating the conditional distribution of It\^{o} processes have been developed \cite{Gordon1993PF, Cai1995adaptive, Lototsky1997SpectralApproach, Brigo1999ProjectionOnExp}. Recently, interest has been expressed in an efficient method for approximating the filtering problem for \acp{SRN}.

An \ac{SRN} describes the time evolution of a set of species/agents through reactions and is found in a wide range of applications, such as biochemical reactions, epidemic processes \cite{brauer2012mathematical, anderson2015stochastic}, and transcription and translation in genomics and virus kinetics \cite{srivastava2002stochastic, hensel2009stochastic} (see Section~\ref{subsec:SRNs} for a brief mathematical introduction and \cite{ben2020hierarchical, munker2024generic} for more details, and \cite{goutsias2013markovian} for a broader overview of applications). 

Mathematically, an \ac{SRN} is a continuous-time Markov chain with a countable state space. The filtering problem for \acp{SRN} is nonlinear, and the solution is generally infinite-dimensional; however, one can derive efficient methods for this problem under additional assumptions regarding the distribution shape or the process dynamics. In \cite{Wiederanders2022AutomatedGeneration}, the authors proposed a general framework for deriving equations for conditional moments for several distribution forms. Moreover, the equations for parameters of the exponential family of distributions were also obtained in \cite{Koyama2016ProjectionBased} based on the drift-diffusion approximation. In \cite{Golightly2011particleMCMC}, the authors also applied Langevin approximation to derive the Markov chain Monte Carlo particle method for the filtering problem. In this context, one can further approximate the \ac{SRN} by a linear system, enabling application of the well-known Kalman filter \cite{Folia2018LNAwithKF}. These methods rely on simplifying the stochastic model and introduce bias, making them unsuitable for our framework. %Since these methods require additional assumptions on the stochastic model, they cannot guarantee (asymptotically) unbiased estimation. In this paper, we will refrain from making such limiting assumptions.

To account for the complex dynamics of \acp{SRN}, several simulation-based methods called \acp{PF} have been introduced for filtering with noisy observations \cite{Fang2022PF, Fang2023RPF}. This class of methods estimates the \ac{QOI} as an average of the simulated paths, weighted according to their likelihood given the observed trajectory, and employs resampling for numerical stability. There are also versions of the \ac{PF} for the case of noise-free observations \cite{Rathinam2021PFwithExactState, rathinam2024stochastic}. As an alternative to the \ac{PF}, the authors of \cite{DAmbrosio2022FFSP} developed the \acf{FFSP} method for approximating the conditional distribution as a solution to the filtering equation on a truncated state space. The filtering equation characterizes the evolution of the conditional expectation and is structurally similar to the \ac{CME}, so the complexity of its solution is comparable to the corresponding methods for \ac{CME}. The limitation of these methods is that they are computationally expensive in high dimensions (i.e., for \acp{SRN} with many hidden species). This work presents two approaches to dimensionality reduction to overcome this difficulty: \ac{UMP} and \ac{CMP} filters.

This work focuses on the filtering problem with noise-free and continuous-in-time observations, which was addressed in \cite{Rathinam2021PFwithExactState, Duso2018SelectedNodeSSA, Fang2022CMEmodularization}. We consider a $d$-dimensional \ac{SRN} $\boldsymbol{Z}$ with initial distribution $\boldsymbol{Z}(0) \sim \mu$ and split the state vector as follows:
$$
    \boldsymbol{Z}(t) = \begin{bmatrix}\boldsymbol{X}(t) \\ \boldsymbol{Y}(t) \end{bmatrix}, \quad t \in [0, T],
$$
where $\boldsymbol{X}(t)$ is a hidden part and $\boldsymbol{Y}(t)$ is the observed part. The filtering problem is to estimate the marginal distribution of $\boldsymbol{X}(t)$, given all observations accumulated until time $t$:
\begin{equation}
\nonumber
    \pi_{\boldsymbol{y}} (\boldsymbol{x}, t) := \Probcondmu{\boldsymbol{X}(t) = \boldsymbol{x} }{ \boldsymbol{Y}(s) = \boldsymbol{y}(s), s \leq t }.
\end{equation}
Here and further, the subscript $\mu$ emphasizes the dependence on the initial distribution.

The current numerical methods for the filtering problem are computationally expensive for large systems, e.g., when the system has many reactions with high rates or $\dim (\boldsymbol{X})$ is large. In practice, it is often necessary to estimate only the  marginal (potentially one-dimensional) distribution of some entries of $\boldsymbol{X}(t)$ influenced by only a subset of the reactions. For this case, we propose an approach that reduces the effective dimensionality of the filtering problem using \ac{MP}. The central idea of the \ac{MP} method for the filtering is schematically illustrated in Figure~\ref{fig:MP_diagram}.

\begin{figure}[H]
    \centering
    \resizebox{0.9\textwidth}{!}{
    \begin{tikzpicture}[thick]
        % Define styles
        \tikzstyle{process} = [rectangle, draw, fill=green!20, text width=5cm, text centered, minimum height=1.5cm]
        \tikzstyle{state} = [rectangle, draw, fill=blue!20, text width=5cm, text centered, minimum height=1.5cm]
        \tikzstyle{decision} = [ellipse, draw, fill=cyan!20, text width=5cm, text centered, minimum height=1.5cm]
        \tikzstyle{arrow} = [thick, ->, >=stealth]
        \tikzstyle{dashed_arrow} = [thick, dashed, ->, >=stealth]
    
        % Nodes
        \node[state, fill=blue!30] (fullSRN) {\textbf{full-dimensional} SRN\\ $\boldsymbol{Z}(t) \in \mathbb{Z}^{d}$};
        \node[process, right=8cm of fullSRN] (fullFilter) {Solving  \textcolor{red}{\textbf{\underline{$d$-dimensional}}} system of filtering equations};
        \node[decision, below=3cm of fullFilter] (qoi) {\textbf{Marginal conditional distribution}};
        \node[state, below=3cm  of fullSRN, fill=blue!20] (projSRN) {\textbf{Projected SRN}\\ $\bar{\boldsymbol{Z}}'(t) \in \mathbb{Z}^{d'}$, \textcolor{black}{$d' \ll d$}};
        \node[process, right=1cm of projSRN] (reducedFilter) {Solving \textcolor{blue}{\textbf{\underline{$d'$-dimensional}}} system of filtering equations};
    
        % text nodes
        \node (MP) [text centered, below of=fullSRN,yshift=-0.8cm, text width=4cm, font={\bfseries\sffamily}] {\textbf{Markovian \\ Projection}};
        \node (Marg) [text centered, below of=fullFilter, yshift=-0.8cm, font={\bfseries\sffamily}] {\textbf{Marginalization}};

        % Arrows
        \draw [dashed] (fullSRN.south east) -- (projSRN.north);
        \draw [dashed] (fullSRN.south west) -- (projSRN.north);
        \draw [arrow] (fullSRN) -- (fullFilter);
        \draw [arrow] (projSRN) -- (reducedFilter);
        \draw [arrow] (reducedFilter.east) -- (qoi);
        \draw [dashed] (fullFilter.south east) -- (qoi.north);
        \draw [dashed] (fullFilter.south west) -- (qoi.north);
    
        % Curse of dim box
        \node[draw=red!50, line width=0.5mm, rounded corners, minimum height=2.7cm, minimum width=6.2cm, xshift=-0.3cm, yshift=0.4cm, inner sep=0pt] at (fullFilter) {};
        \node (CD) [text centered, above of=fullFilter, xshift=-1cm, yshift=0.3cm] { \textcolor{red}{\textbf{Curse of dimensionality}} };
        
    \end{tikzpicture}
    }
    \caption{
    A graphical illustration of the proposed projection methods for the filtering problem. Instead of solving a full $d$-dimensional system of filtering equations, \ac{MP} methods approximate the \ac{SRN} dynamics by another \ac{SRN} of lower dimensionality $d' \ll d$, allowing to work with significantly smaller system of filtering equations.
    }
    \label{fig:MP_diagram}
\end{figure}
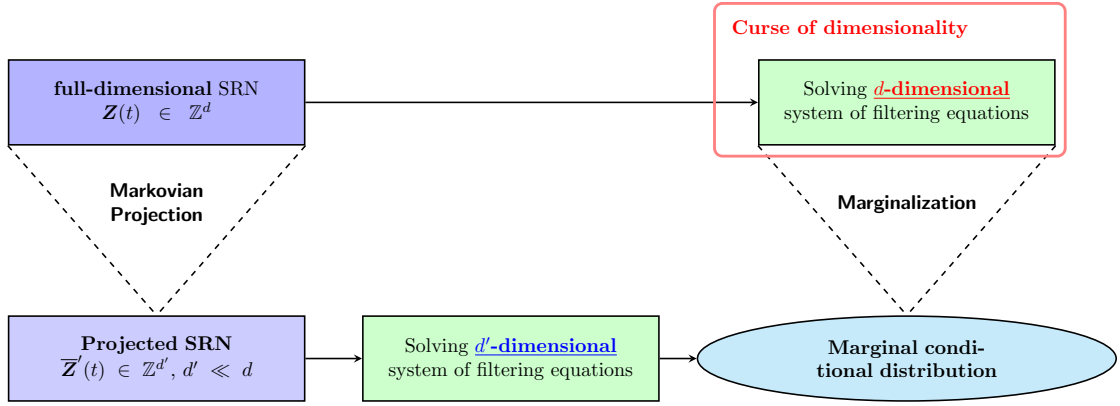

The general idea of \ac{MP} is to mimic the marginal distribution of a multidimensional process via another Markov process of lower dimensionality. This approach was first proposed in \cite{Gyongy1986MP} for It\^o processes and has been employed in many applications \cite{piterbarg2006markovian, Bayer2019Implied, Bentata2009Mimicking}. Recently, this method was also applied to the \acp{SRN} to derive an efficient importance sampling \ac{MC} estimator for rare event probabilities by reducing the dimensionality of the underlying model \cite{Hammouda2023MP}, and to solve the \ac{CME} \cite{ocal2023model}. In contrast to projection methods from statistical physics \cite{hijon2010mori, zhu2018estimation}, the \ac{MP}-based models do not have a computationally challenging memory term that accounts for the non-Markovian part, but they still mimic the marginal distributions exactly.

This work introduces the \acf{FMP} framework for dimensionality reduction in filtering problems for \acfp{SRN}. Our main contributions are:

\begin{enumerate}
    \item We extend the classical \acf{MP} theorem \cite{Hammouda2023MP} from forward problems to filtering problems with exact observations. The resulting \acf{FMP} theorem (Theorem~\ref{th:FMP}) is the first projection result that preserves the marginal conditional distribution of the hidden process.
    \item We formulate two filters for the marginal filtering problem:
    \begin{itemize}
        \item[(i)] a \ac{UMP} filter, a conceptually simple filter based on the direct application of the \ac{MP} theorem; 
        \item[(ii)] a \ac{CMP} filter, derived from the \ac{FMP} theorem. The latter retains exact marginal consistency while enabling low-dimensional computation.
    \end{itemize}
    \item The \ac{CMP} filter combines a \acf{PF} for estimating projected propensities with a \acf{FFSP} for solving the reduced filtering equations. This hybrid design acts as a variance reduction mechanism for \ac{PF}-type estimators, allowing orders of magnitude fewer particles.
    \item We present rigorous error and convergence analysis. We derive the first error decomposition for projection-based filtering (Theorem~\ref{th:sensitivity}) and prove $O(M^{-1/2})$ convergence of the \ac{CMP} filter (Corollary~\ref{corollary:FMP_error_PF}), linking \ac{PF} variance to filtering accuracy.
    \item Computational efficiency and scalability. Extensive numerical experiments on several SRNs demonstrate up to two orders of magnitude speedup over full-dimensional \ac{FFSP} and \ac{PF} methods—the current state-of-the-art approaches for \ac{SRN} filtering while maintaining comparable accuracy in marginal distributions.
\end{enumerate}

Together, these contributions establish the FMP framework as a general, consistent, and computationally efficient approach to filtering in SRNs and related high-dimensional stochastic processes.

The outline of this paper is as follows. The introduction provides the mathematical background of the \ac{SRN} models and discusses the filtering problem. Section~\ref{sec:2_MP} describes the \ac{UMP} filter, which is an adaptation of the standard \ac{MP} approach, and its limitations in the filtering problem. Next, Section~\ref{sec:3_FMP} extends these ideas and presents a novel \ac{FMP} theorem as a natural generalization of \ac{MP} and introduces the \ac{CMP} filter. It also provides an error analysis of the novel filter, followed by two numerical examples and a final discussion.
        \subsection{Stochastic Reaction Networks}
\label{subsec:SRNs}
Consider a chemical system with $d$ interacting species $S_1, \dots, S_d$ and $J$ reactions given by 
\begin{equation}
    \sum_{i=1}^{d} \nu_{ij}^{-} S_i \longrightarrow \sum_{i=1}^{d} \nu_{ij}^{+} S_i, \quad \text{for } j = 1, \dots, J, 
\end{equation}
where $\nu_{ij}^{-}$ is the number of molecules of $S_i$ consumed by reaction $j$, and $\nu_{ij}^{+}$ is the number of molecules of $S_i$ produced by the reaction $j$. Let $\boldsymbol{Z}(t) = \left( Z_1(t), \dots, Z_d(t) \right)^\top \in \mathsf{Z} \subseteq \mathbb{Z}_{\geq 0}^d$ be the copy numbers (amount of molecules) of each species at time $t$.

In a low copy number regime, stochastic effects dominate, and the system dynamics can be modeled as a continuous-time Markov chain \cite{Gillespie1976, Gillespie1992DerivationCME} with transition probabilities between times $t$ and $t+h \, (h>0)$, given by
\begin{equation}
    \Probcond{ \boldsymbol{Z}(t+h) = \boldsymbol{z} + \boldsymbol{\nu}_j }{ \boldsymbol{Z}(t) = \boldsymbol{z} } = a_j (\boldsymbol{z})h + o(h), 
\end{equation}
where $\boldsymbol{\nu}_j := \left( \nu_{1j}^{+} - \nu_{1j}^{-}, \dots,  \nu_{dj}^{+} - \nu_{dj}^{-}  \right)^\top$ is a stoichiometric vector, and $a_j: \mathsf{Z} \to \mathbb{R}_{\geq 0}$ is the \textit{propensity function} of reaction $j$. For chemical reactions, the propensities are typically given by the \textit{mass action kinetics}: 
\begin{equation}
    \label{eq:mass_action}
    a_j (\boldsymbol{z}) = 
    \begin{dcases}
        \theta_j \prod_{i=1}^{d} \frac{z_i !}{(z_i - \nu_{ij}^{-})!} & \text{if } \forall i \; z_i \geq \nu_{ij}^{-} \\
        0 & \text{otherwise,} \\
    \end{dcases} 
\end{equation}
where the constant $\theta_j$ is the \textit{reaction rate} of reaction $j$. This is not the only possible form of propensity functions, and the results in this work can be directly extended to other types of propensity functions (e.g., Hill-type propensities \cite{murray2002mathematical}).

Using the random time change representation \cite{Ethier1986MarkovProcess}, one can express the current state of the process $\boldsymbol{Z}$ via mutually independent unit-rate Poisson processes $R_1, \dots, R_J$:
\begin{equation}
\label{eq:RTC_representation}
    \boldsymbol{Z}(t) = \boldsymbol{Z}(0) + \sum_{j=1}^{J} R_j \left( \int_0^t  a_j \left( \boldsymbol{Z}(s) \right) \diff s \right)  \boldsymbol{\nu}_j ,
\end{equation}
where $\boldsymbol{Z}(0)$ is a random vector (independent of $R_1, \dots, R_J$) characterized by the initial distribution $\mu$. 

A fundamental problem associated with the \ac{SRN} is to determine $p (\boldsymbol{z}, t) := \Probmu{\boldsymbol{Z}(t) = \boldsymbol{z}}$ for a given initial distribution $\mu$ and a set of reactions $\{ a_j, \boldsymbol{\nu}_j \}, \, j = 1,\dots,J$. This forward problem can be addressed using \ac{MC} methods, for example, a \acf{SSA} \cite{Gillespie1977SSA}. This approach involves sampling exponential waiting times between reactions and recomputing the propensities after each. Another approach is to approximate the solution of the \acf{CME} \cite{Gillespie1992DerivationCME}, which describes the time evolution of the \ac{PMF} $p(\boldsymbol{z}, t)$.

This work focuses on the filtering problem, which is structurally similar to the forward problem. Therefore, the numerical methods for its solution are closely related to the corresponding methods for the forward problem.

\subsection{Filtering Problem}
\label{subsec:filtering_problem}
Consider a noise-free filtering problem with exact observations, where the current state $\boldsymbol{Z}(t)$ at time $t \in [0, T]$ can be split as follows:
$$
    \boldsymbol{Z}(t) = \begin{bmatrix} \boldsymbol{X}(t) \\ \boldsymbol{Y}(t) \end{bmatrix},
$$
where $\boldsymbol{X}(t) \in \mathsf{X}$ is the hidden process corresponding to the unobserved species, and $\boldsymbol{Y}(t) \in \mathsf{Y}$ is the observed process corresponding to the species that can be tracked. For continuous and noise-free observations, the filtering problem is to estimate the conditional distribution of the unobserved process:
\begin{equation}
\label{eq:pi_def}    
    \pi_{\boldsymbol{y}} (\boldsymbol{x}, t) := \Probcondmu{\boldsymbol{X}(t) = \boldsymbol{x} }{ \boldsymbol{Y}(s) = \boldsymbol{y}(s), s \leq t }
\end{equation}
for a given trajectory $\{ \boldsymbol{y}(s), s \leq t \}$. 

Following \cite{DAmbrosio2022FFSP}, we introduce the following non-explosivity condition to ensure the uniqueness of $\pi_{\boldsymbol{y}}$:
\begin{equation}
\label{eq:non_expl_cond} %\tag{A1}
    \sup_{s \in [0, T]} \sum_{j=1}^{J} \Emu{a_j^2 (\boldsymbol{Z}(s))} < \infty.
\end{equation}
This condition means the system state does not increase to infinity (almost surely) in a finite time interval $[0, T]$. This assumption holds for most of the \acp{SRN} considered in the literature.

We split the stoichiometric vectors $\boldsymbol{\nu}_j$ into parts $\boldsymbol{\nu}_{\boldsymbol{x}, j}$ and $\boldsymbol{\nu}_{\boldsymbol{y}, j}$, corresponding to $\boldsymbol{X}$ and $\boldsymbol{Y}$, respectively. Let us denote by $\mathcal{O} := \{ j : \boldsymbol{\nu}_{\boldsymbol{y}, j} \neq \boldsymbol{0} \}$ the set of observable reactions that can alter $\boldsymbol{Y}$ and denote by $\mathcal{U} := \{ j : \boldsymbol{\nu}_{\boldsymbol{y}, j} = \boldsymbol{0} \}$ the set of unobservable reactions that cannot alter $\boldsymbol{Y}$. For a given trajectory $\{ \boldsymbol{y}(s), s \leq T \}$, the corresponding jump times are denoted by $t_1, \dots, t_N$.

Under the non-explosivity assumption \eqref{eq:non_expl_cond}, the conditional distribution $\pi_{\boldsymbol{y}}$ at each of the intervals $(t_k, t_{k+1})$ evolves according to the following \ac{ODE} \cite{DAmbrosio2022FFSP, Rathinam2021PFwithExactState}:
\begin{equation}
    \label{eq:filtering_equation_pi}
    \begin{aligned}
        \frac{\mathrm{d}}{\mathrm{d} t} \pi_{\boldsymbol{y}} (\boldsymbol{x}, t) = &\sum_{j \in \mathcal{U}} \pi_{\boldsymbol{y}} (\boldsymbol{x} - \boldsymbol{\nu}_{\boldsymbol{x}, j}, t) a_j \left( \boldsymbol{x} - \boldsymbol{\nu}_{\boldsymbol{x}, j}, \boldsymbol{y}(t_k) \right) - \sum_{j \in \mathcal{U}} \pi_{\boldsymbol{y}}(\boldsymbol{x}, t) a_j (\boldsymbol{x}, \boldsymbol{y}(t_k)) \\
        &- \pi_{\boldsymbol{y}} (\boldsymbol{x}, t) \cdot \sum_{j \in \mathcal{O}} \left( a_j (\boldsymbol{x}, \boldsymbol{y}(t_k)) -  \sum_{\hat{\boldsymbol{x}} \in \mathsf{X}} a_j (\hat{\boldsymbol{x}}, \boldsymbol{y}(t_k)) \pi_{\boldsymbol{y}} (\hat{\boldsymbol{x}}, t)\right).
    \end{aligned}
\end{equation}

At each jump point $t_k$, we know exactly how the state vector changed and therefore can identify the set of reactions that might have caused the jump: $\mathcal{O}_k := \{ j : \boldsymbol{\nu}_{\boldsymbol{y}, j} = \boldsymbol{y}(t_k) - \boldsymbol{y}(t_k^{-}) \}$. At jump times $t_k$, the conditional probability satisfies the following:
\begin{equation}
    \label{eq:filtering_equation_pi_jump}
    \pi_{\boldsymbol{y}} (\boldsymbol{x}, t_k) = \frac{ 
    \sum\limits_{j \in \mathcal{O}_k} a_j (\boldsymbol{x} - \boldsymbol{\nu}_{\boldsymbol{x}, j}, \boldsymbol{y}(t_{k-1})) \pi_{\boldsymbol{y}} (\boldsymbol{x} - \boldsymbol{\nu}_{\boldsymbol{x}, j}, t_{k}^{-})
    }{ 
    \sum\limits_{\hat{\boldsymbol{x}} \in \mathsf{X}} \sum\limits_{j \in \mathcal{O}_k} a_j (\hat{\boldsymbol{x}}, \boldsymbol{y}(t_{k-1})) \pi_{\boldsymbol{y}} (\hat{\boldsymbol{x}}, t_k^{-})
    }.
\end{equation}

The pair of equations \eqref{eq:filtering_equation_pi} and \eqref{eq:filtering_equation_pi_jump}, called the \textit{filtering equations}, characterize the complete dynamics of $\pi_{\boldsymbol{y}}$. The initial condition 
\begin{equation}
\label{eq:pi0_def}
    \pi_{\boldsymbol{y}} (\boldsymbol{x}, 0) = \Probcond{\boldsymbol{X}(0) = \boldsymbol{x}}{\boldsymbol{Y}(0) = \boldsymbol{y}(0)} 
\end{equation} 
can be derived from the initial distribution $\mu$, which is assumed to be given.

Due to the non-linearity of \eqref{eq:filtering_equation_pi} and \eqref{eq:filtering_equation_pi_jump} it is more convenient to introduce an unnormalized \ac{PMF} $\rho_{\boldsymbol{y}} \propto \pi_{\boldsymbol{y}}$ that obeys the following equations:
\begin{equation}
\label{eq:filtering_equation_rho}
\begin{aligned}
    &\text{for } t \in (t_k, t_{k+1}) : \\
    &\frac{\mathrm{d}}{\mathrm{d} t} \rho_{\boldsymbol{y}} (\boldsymbol{x}, t) = \sum_{j \in \mathcal{U}} \rho_{\boldsymbol{y}} (\boldsymbol{x} - \boldsymbol{\nu}_{\boldsymbol{x}, j}, t) a_j (\boldsymbol{x}-\boldsymbol{\nu}_{\boldsymbol{x}, j}, \boldsymbol{y}(t_k))
    - \sum_{j=1}^{J} \rho_{\boldsymbol{y}} (\boldsymbol{x}, t) a_j (\boldsymbol{x},\boldsymbol{y}(t_k)),
\end{aligned}
\end{equation}

\begin{equation}
\begin{aligned}
\label{eq:filtering_equation_rho_jump}
    &\text{for } t = t_k : \\
    &\rho_{\boldsymbol{y}} (\boldsymbol{x}, t_k) =  \frac{1}{\abs{\mathcal{O}_k}} \sum\limits_{j \in \mathcal{O}_k} a_j (\boldsymbol{x} - \boldsymbol{\nu}_{\boldsymbol{x}, j}, \boldsymbol{y}(t_{k-1})) \rho_{\boldsymbol{y}} (\boldsymbol{x} - \boldsymbol{\nu}_{\boldsymbol{x}, j}, t_{k}^{-}),
\end{aligned}
\end{equation}
with the initial condition $\rho_{\boldsymbol{y}}(\boldsymbol{x}, 0) = \pi_{\boldsymbol{y}}(\boldsymbol{x}, 0)$.  
After reaching a solution to \eqref{eq:filtering_equation_rho} and \eqref{eq:filtering_equation_rho_jump},  $\pi_{\boldsymbol{y}}$ can be obtained with normalization:
\begin{equation}
\label{eq:rho_normalization_full}
    \pi_{\boldsymbol{y}}(\boldsymbol{x}, t) = \frac{\rho_{\boldsymbol{y}} (\boldsymbol{x}, t)}{\sum\limits_{\hat{\boldsymbol{x}} \in \mathsf{X}} \rho_{\boldsymbol{y}} (\hat{\boldsymbol{x}}, t)} .
\end{equation}

The filtering equation \eqref{eq:filtering_equation_pi} should be solved simultaneously for all possible states $\boldsymbol{x} \in \mathsf{X}$, leading to a coupled system of possibly an infinite number of equations. Thus, it requires numerical approximation. The same applies to the unnormalized version \eqref{eq:filtering_equation_rho}.

The following section briefly summarizes the two methods for numerically approximating \eqref{eq:pi_def}.

\subsubsection{Filtered Finite State Projection}
The \acf{FFSP} method \cite{DAmbrosio2022FFSP} is based on truncating the state space, similar to the idea introduced in \cite{Munsky2006FSP} for solving the \acf{CME}. Let us consider a finite subset of hidden states $\mathsf{X}_N \subset \mathsf{X}$ with $ \abs{\mathsf{X}_N} = N < \infty$ and assume that the probability of visiting the remaining states at the time interval $[0,T]$ is negligible. 

By setting $\rho_{\boldsymbol{y}}(\hat{\boldsymbol{x}}, \cdot) = 0$ for all $\hat{\boldsymbol{x}} \in \mathsf{X} \setminus \mathsf{X}_N$, Equation \eqref{eq:filtering_equation_rho} becomes a system of $N$ linear \acp{ODE}, which can be solved analytically or numerically. 

To make the error from the state space truncation computable, it is necessary to assume that the observed propensities are bounded for each $\boldsymbol{y}$. That is, there exists a function $C_1 : \mathsf{Y} \to \mathbb{R}_{\geq 0}$ such that
\begin{equation}
    \label{eq:assumption_bound_propens} %\tag{A2}
    \sup_{\boldsymbol{x} \in \mathsf{X}} \sum_{j \in \mathcal{O}} a_j \left( \boldsymbol{x}, \boldsymbol{y} \right) \leq C_1(\boldsymbol{y}) .
\end{equation}

The error of the \ac{FFSP} method can be controlled by selecting a truncated space $\mathsf{X}_N$ according to \cite[Theorems~2 and 4]{DAmbrosio2022FFSP}. More precisely, the error can be reduced by adaptively including more states in $\mathsf{X}_N$ and including more equations in the \ac{FFSP} system (for more detailed numerical algorithms and the error analysis of this method, refer to \cite{DAmbrosio2022FFSP}).

The main issue of the \ac{FFSP} method is that once the number of states $N = \abs{\mathsf{X}_N}$ is too large, solving the system \eqref{eq:filtering_equation_rho} becomes computationally expensive. In particular, for a $n$-dimensional truncated state space with $m$ states in each direction, the total number of states is $N = m^n$, that is, it increases exponentially with respect to the dimensionality $n$ of the hidden state space.

\subsubsection{Particle Filter}
\label{subsubsec:PF}
Another method for the filtering problem, called the \ac{PF}, was introduced in \cite{Rathinam2021PFwithExactState}. The idea is to sample a pair of processes $\boldsymbol{V}$ and $w$ such that 
\begin{equation}
    \label{eq:PF_expectation_rho}
    \rho_{\boldsymbol{y}} (\boldsymbol{x}, t) = \Emu{ 1_{\{ \boldsymbol{V}(t) = \boldsymbol{x} \}} w(t) },
\end{equation}
where $1_{\{\cdot\}}$ denotes the indicator function defined as follows:
$$
    1_{A} (\omega) = 
    \begin{cases}
        1 & \text{if } \omega \in A \\
        0 & \text{if } \omega \not \in A \\
    \end{cases} .
$$

The process $\boldsymbol{V}$ evolves according to the original \ac{SRN} with all observable reactions removed with additional reactions at jump times $t_k$:
\begin{equation}
    \label{eq:PF_process_V}
    \boldsymbol{V}(t) = \boldsymbol{X}(0) + \sum_{j\in\mathcal{U}} R_j \left( \int_{0}^t  a_j \left( \boldsymbol{V}(s), \boldsymbol{y}(t) \right) \diff s \right) \boldsymbol{\nu}_{\boldsymbol{x},j} + \sum_{t_k \leq t} \boldsymbol{\nu}_{\boldsymbol{x}, \ell_k},
\end{equation}
where index $\ell_k$ is a uniform random draw from the set $\mathcal{O}_k$. The last term in \eqref{eq:PF_process_V} adjusts for the observed jumps in trajectory $\boldsymbol{y}$.
    
The process $w$ represents a weight of the trajectory of $\boldsymbol{V}$ given observations $\{ \boldsymbol{y}(s), s \leq t \}$. It is initialized as $w(0) = 1$ and can be computed exactly given $\boldsymbol{V}$:
\begin{equation}
    \label{eq:PF_process_w}
        w(t) = w(0) \exp \left[ - \sum_{j \in \mathcal{O}} \int_{0}^t a_j(\boldsymbol{V}(s), \boldsymbol{y}(t_k)) \diff s \right] \cdot \prod_{t_k \leq t} a_{\ell_k} (\boldsymbol{V}(t_k^{-}), \boldsymbol{y}(t_k^{-})).
\end{equation}
The exponential term corresponds to the likelihood of not observing reactions $\mathcal{O}$ in between jumps, and the last term adjusts for the observed jumps. 

Following \cite{Rathinam2021PFwithExactState}, we use  \ac{SSA} to simulate $\boldsymbol{V}$ in between jumps $t \in [t_k, t_{k+1})$ and draw a random reaction index $j_{k+1}$ to fire at jumps $t=t_{k+1}$. A trajectory of $w$ is computed exactly based on \eqref{eq:PF_process_w}.

Let $(\boldsymbol{V}_{i}, w_{i}), \, i = 1, \dots, M$ be independent realizations (particles) of $(\boldsymbol{V}, w)$, then 
\begin{equation}
    \label{eq:PF_weighed_avg}
    \rho_{\boldsymbol{y}} (\boldsymbol{x}, t) \approx \frac{1}{M} \sum_{i=1}^{M} 1_{\left\{ \boldsymbol{V}_{i}(t) = \boldsymbol{x} \right\}} w_{i}(t).
\end{equation}

In practice, only a few particles significantly contribute to the weighted average \eqref{eq:PF_weighed_avg} even using many samples, (i.e., the effective sample size can be critically small). This problem is known as sample collapse and occurs because, for most trajectories, the weight process $w$ rapidly decreases to zero. To avoid this problem, one can use the \textit{resampling} procedure, discarding particles with small weights and multiplying the number of particles with large weights (e.g., using a bootstrap algorithm \cite{Gordon1993PF}). However, the resampling step breaks particle independence and introduces additional variance to the resulting estimator. Therefore, a trade-off exists between the error introduced by resampling and the error from weight degeneracy. This problem becomes significant for high-dimensional processes \cite{Snyder2008HighDimPF, Djuric2013HighDimPF}. Refer to \cite{Bain2009Fundamentals} for a more detailed discussion on this problem and other resampling algorithms. 

In addition, a common problem of \ac{MC}-based methods is related to high variance. The indicator function in \eqref{eq:PF_weighed_avg} is prone to this problem if $\boldsymbol{x}$ is a low-probability state (e.g., if $\boldsymbol{x}$ is in the tail of the distribution). In contrast, if we are not interested in the distribution itself but in the conditional expectation of $f(\boldsymbol{X}(t))$ with a given function $f$, then the \ac{PF} typically outperform other methods.

\subsection{Marginal Filtering Problem}
\label{sec:margin_filtering_problem}
In this paper, we focus on situations in which only the distribution of a small subset of the hidden species should be estimated. The hidden state vector is split as follows: $ \boldsymbol{X}(t) = \begin{bmatrix} \boldsymbol{X}'(t) \\ \boldsymbol{X}''(t) \end{bmatrix} $, where $\boldsymbol{X}'(t) \in \mathsf{X}'$ corresponds to the hidden species of interest, and $\boldsymbol{X}''(t) \in \mathsf{X}''$ corresponds to the remaining hidden species. The goal of the marginal filtering problem is to estimate the following:
\begin{equation}
\label{eq:filtering_problem_margin}
    \pi'_{\boldsymbol{y}} (\boldsymbol{x}', t) = \Probcondmu{\boldsymbol{X}'(t) = \boldsymbol{x}' }{ \boldsymbol{Y}(s) = \boldsymbol{y}(s), s \leq t },
\end{equation}
for a given trajectory $\{ \boldsymbol{y}(s), s \leq t \}$. Clearly,  $\pi'_{\boldsymbol{y}}$ can be derived from $\pi_{\boldsymbol{y}}$ via marginalization:
\begin{equation}
\nonumber    
    \pi'_{\boldsymbol{y}} ({\boldsymbol{x}'}, t) = \sum_{{\boldsymbol{x}''} \in \mathsf{X}''} \pi_{\boldsymbol{y}} \left( \begin{bmatrix} \boldsymbol{x}' \\ \boldsymbol{x}'' \end{bmatrix}, t \right).
\end{equation}
However, estimating $\pi_{\boldsymbol{y}}$ entails the curse of dimensionality and significant computational cost for solving the high-dimensional systems \eqref{eq:filtering_equation_pi} and \eqref{eq:filtering_equation_pi_jump}. Therefore, we aim to exclude species from $\boldsymbol{X}''$ and solve the filtering problem only for the $d'$-dimensional process
$$
    \boldsymbol{Z}'(t) := 
    \begin{bmatrix}
        \boldsymbol{X}'(t) \\
        \boldsymbol{Y}(t)
    \end{bmatrix} .
$$
This process is coupled with the $d''$-dimensional process $\boldsymbol{Z}''(t) := \boldsymbol{X}''(t)$, complicating the analysis of $\boldsymbol{Z}'$ as a separate process. The overall splitting of the process $\boldsymbol{Z}$ is summarized in the diagram:

\begin{center}
\begin{tikzpicture}[node distance=3cm, auto]
    \node (Z) [align=left]{$\boldsymbol{Z} =
        \begin{bmatrix}  
            \boldsymbol{X}' \\ 
            \boldsymbol{X}'' \\ 
            \boldsymbol{Y} \\ 
        \end{bmatrix}$};
    \node (Z') [align=left, right of=Z] {$\begin{bmatrix}
            \boldsymbol{X}' \\
            \boldsymbol{Y}
        \end{bmatrix}
        = \boldsymbol{Z}'$};
    \node (Z'') [align=left, below of=Z', node distance=1cm] {$\phantom{'}\begin{bmatrix} \boldsymbol{X}'' \end{bmatrix} = \boldsymbol{Z}''$};
    
    \draw[->, thick] (Z.east) -- (Z'.west);
    \draw[->, thick] (Z.east) -- (Z''.west);

\end{tikzpicture}
\end{center}

To clarify the difficulty of treating $\boldsymbol{Z}'$ as a distinct process separate from $\boldsymbol{Z}''$, we consider the random time change representation \eqref{eq:RTC_representation} for it:
\begin{equation}
\label{eq:RTC_splited}
\begin{aligned}
    \boldsymbol{Z}'(t)  &= \boldsymbol{Z}'(0) + \sum_{j=1}^{J} R_j \left( \int_0^t  a_j \left( \begin{bmatrix} \boldsymbol{Z}'(t) \\ \boldsymbol{Z}''(t)\end{bmatrix} \right) \diff s \right)  \boldsymbol{\nu}_j' , \\ 
    \boldsymbol{Z}''(t)  &= \boldsymbol{Z}''(0) + \sum_{j=1}^{J} R_j \left( \int_0^t  a_j \left( \begin{bmatrix} \boldsymbol{Z}'(t) \\ \boldsymbol{Z}''(t)\end{bmatrix} \right) \diff s \right)  \boldsymbol{\nu}_j'' , 
\end{aligned}
\end{equation}
where $\boldsymbol{\nu}_j'$ and $\boldsymbol{\nu}_j''$ are the corresponding parts of the stoichiometric vector $\boldsymbol{\nu}_j$. The process $\boldsymbol{Z}'$ can be considered as an \ac{SRN} with fewer species, but its propensities are random (due to $\boldsymbol{Z}''(t)$). Moreover, $\boldsymbol{Z}'$ is not a Markov process because $\boldsymbol{Z}''(t)$ depends on the past states of $\boldsymbol{Z}'$, which prevents applying classical filtering approaches. 

The following two sections discuss constructing a $d'$-dimensional Markov process that mimics $\boldsymbol{Z}'$, enabling solving the marginal filtering problem more efficiently.

    \section{The Unconditional Markovian Projection Filter}
        \label{sec:2_MP}
        We present the \acf{UMP} filter, which is based on the MP introduced for forward problems. \ac{MP} is a dimensionality reduction approach that preserves the marginal (unconditional) distribution while resulting in a Markovian process. It was originally introduced for It\^{o} processes \cite{Gyongy1986MP} and recently adapted to \acp{SRN} \cite{Hammouda2023MP}. This approach is naive for the filtering settings because it applies the \ac{MP} framework for projected propensities, ignoring the observed trajectory. The aim is to construct a $d'$-dimensional process $\bar{\boldsymbol{Z}}'$ with the same marginal (in time) distribution as $\boldsymbol{Z}'$. The crucial point is that this surrogate $\bar{\boldsymbol{Z}}'$ (unlike $\boldsymbol{Z}'$) is Markovian, allowing the filtering problem to be solved with classical methods.

Let us recall the (unconditional) \ac{MP} for \acp{SRN} \cite{Hammouda2023MP} in the following Theorem.

\begin{theorem}[Markovian projection for \acp{SRN}]
    \label{th:MP}
    Let $\boldsymbol{Z}(t) = \begin{bmatrix} \boldsymbol{Z}'(t) \\ \boldsymbol{Z}''(t)\end{bmatrix}$ be a non-explosive (i.e., satisfying~\ref{eq:non_expl_cond}) \ac{SRN} with initial distribution $\mu$. 
    A $d'$-dimensional stochastic process $\bar{\boldsymbol{Z}}'$ via independent Poisson processes $\bar{R}_1, \dots, \bar{R}_J$ is defined as follows:
    \begin{equation}
    \label{eq:MP_model}
        \bar{\boldsymbol{Z}}'(t) = \bar{\boldsymbol{Z}}'(0) + \sum_{j=1}^{J}  \bar{R}_j \left( \int_0^t \bar{a}_j \left( \bar{\boldsymbol{Z}}'(s), s \right) \diff s \right) \boldsymbol{\nu}_{j}' , \quad t \in [0,T]
    \end{equation}
    with 
    \begin{equation}
    \label{eq:MP_a_bar_def}
        \bar{a}_j ({\boldsymbol{z}}', t) := \Econdmu{ a_j(\boldsymbol{Z}(t)) }{ \boldsymbol{Z}'(t) = {\boldsymbol{z}}'}, \quad {\boldsymbol{z}}' \in \mathbb{Z}_{\geq 0}^{d'}
    \end{equation}
    and $\bar{\boldsymbol{Z}}'(0) \overset{d}{=} \boldsymbol{Z}'(0)$. \footnote{The symbol $\overset{d}{=}$ denotes the equality in distribution.} Then, $\bar{\boldsymbol{Z}}'(t)$ has the same distribution as $\boldsymbol{Z}'(t)$ for all $t \in [0, T]$.
\end{theorem}

\begin{proof}
    One can derive the statement from \cite[Theorem 3.1]{Hammouda2023MP}. Appendix~\ref{sec:MP_proof} also provides an alternative proof using the \ac{CME}.
\end{proof}

Theorem~\ref{th:MP} allows the construction of another \ac{SRN} $\bar{\boldsymbol{Z}}'$ containing only the species necessary for the filtering problem \eqref{eq:filtering_problem_margin} while preserving the marginal (in time) distribution. Similar to $\boldsymbol{Z}'$, we split the \ac{MP} process $\bar{\boldsymbol{Z}}' = \begin{bmatrix} \bar{\boldsymbol{X}}' \\  \bar{\boldsymbol{Y}}' \end{bmatrix}$ and define the \ac{MP} filtering problem as follows:
\begin{equation}
    \label{eq:pi_bar_def}
    \bar{\pi}'_{\boldsymbol{y}} (\boldsymbol{x}', t) := \Probcondmu{\bar{\boldsymbol{X}}'(t) = \boldsymbol{x}' }{ \bar{\boldsymbol{Y}}(s) = \boldsymbol{y}(s), s \leq t }.
\end{equation}
As $\dim{\bar{\boldsymbol{X}}'} = \dim{\boldsymbol{X}}' <  \dim{\boldsymbol{X}}$, it is expected that the \ac{MP} filtering problem can be solved much more efficiently than the original problem. In particular, for the \ac{FFSP} method, the truncated state space $\mathsf{X}'_N$ of $\bar{\boldsymbol{X}}'$ contains significantly fewer states than $\mathsf{X}_N$, and hence, fewer equations. 

Moreover, the projected \ac{SRN} may have fewer reactions because some $\nu_j'$ may be null vectors, especially when  $\dim{\boldsymbol{X}}' \ll  \dim{\boldsymbol{X}}$. Without loss of generality, we denote $\{ 1, \dots, J' \}$ with $J' \leq J$ the indices of reactions in the projected \ac{SRN}. All observed reactions from $\mathcal{O}$ are preserved, and the set of projected hidden reactions $\mathcal{U}'$ may be smaller than the set $\mathcal{U}$ for the full-dimensional system.

To estimate projected propensities \eqref{eq:MP_a_bar_def}, we use a \ac{MC} estimator (see Appendix~\ref{sec:MP_appendix} for more details). One can also use the $L^2$ regression method \cite{Hammouda2023MP}, which is especially efficient if some structural features of the functions $\bar{a}_j$ are known.

New propensities $\{\bar{a}_j\}_{j=1}^{J'}$ depend not only on a current state ${\boldsymbol{z}}'$ but also on time $t$, which can introduce some challenges. Algorithms adapted to time-dependent propensity functions (e.g., the modified next reaction method \cite[Section~5]{Anderson2007ModifiedNextReaction}) must be applied to simulate $\bar{\boldsymbol{Z}}'$.

The main issue with the standard \ac{MP} for the filtering problem is that it does not guarantee that $\bar{\pi}'_{\boldsymbol{y}}$ defined in \eqref{eq:pi_bar_def} equals to the distribution of interest $\pi'_{\boldsymbol{y}}$ from \eqref{eq:filtering_problem_margin} because the filtering problem has a condition on the past process states, whereas the \ac{MP} theorem states only that marginal (in time) distributions of $\bar{\boldsymbol{Z}}'(t)$ and $\boldsymbol{Z}'(t)$ coincide for any fixed $t$. However, we still can define a filter based on the standard \ac{MP}, and as we show in Section~\ref{sec:4_Examples}, the bias of this filter may be sufficiently small in some cases. The overall scheme for the \ac{UMP} filter based on the presented standard \ac{MP} theorem is given in Appendix~\ref{sec:MP_appendix}.

The following section proposes a new \ac{MP} method explicitly designed for the filtering problem to resolve this inconsistency.

    \section{The Consistent Conditional Markovian Projection Filter}
        \label{sec:3_FMP}
        In this section, we state a novel \ac{MP}-type theorem that is adapted to the filtering problem and design the corresponding consistent filtering algorithm. Our modification introduces an additional conditioning on the observed trajectory within the projected propensities, resulting in another surrogate \ac{SRN} that yields the desired marginal conditional distributions. The result is summarized in the following theorem, which can be considered an extension of Theorem~\ref{th:MP} for the filtering problem.
 
\begin{theorem}[\acf{FMP} for \acp{SRN}]
    \label{th:FMP}
    Let 
    $\boldsymbol{Z} = \begin{bmatrix} \boldsymbol{X}' \\ \boldsymbol{X}'' \\ \boldsymbol{Y} \end{bmatrix}$ 
    be a non-explosive $d$-dimensional \ac{SRN} with the initial distribution $\mu$ and $\boldsymbol{Z}'(t) := \begin{bmatrix} \boldsymbol{X}'(t) \\ \boldsymbol{Y}(t)  \end{bmatrix} \in \mathbb{Z}_{\geq 0}^{d'} $. 
    The $d'$-dimensional stochastic process $\Tilde{\boldsymbol{Z}}'(t) = \begin{bmatrix} \Tilde{\boldsymbol{X}}'(t) \\ \Tilde{\boldsymbol{Y}}(t)  \end{bmatrix} $ is defined via independent Poisson processes $\Tilde{R}_1, \dots, \Tilde{R}_J$ as follows:
    \begin{equation}
        \label{eq:FMP_model}
        \Tilde{\boldsymbol{Z}}'(t) = \Tilde{\boldsymbol{Z}}'(0) + \sum_{j=1}^{J} \Tilde{R}_j \left( \int_0^t \Tilde{a}_j \left( \Tilde{\boldsymbol{Z}}'(s), s \right) \diff s \right) \boldsymbol{\nu}_{j}' , \quad t \in [0,T]
    \end{equation}
    with 
    \begin{equation}
    \label{eq:FMP_a_tilde_def}
        \Tilde{a}_j ({\boldsymbol{z}}', t) := \Econdmu{ a_j(\boldsymbol{Z}(t)) }{ \boldsymbol{Z}'(t) = {\boldsymbol{z}}', \boldsymbol{Y}(s) = \boldsymbol{y}(s) , s \leq t}
    \end{equation}
    and $\Tilde{\boldsymbol{Z}}'(0) \overset{d}{=} \boldsymbol{Z}'(0)$. Then, the distribution of $\Tilde{\boldsymbol{Z}}'(t)$ conditioned on $\{ \Tilde{\boldsymbol{Y}}(s) = \boldsymbol{y}(s), s \leq t \}$ is the same as the distribution of $\boldsymbol{Z}'(t)$ conditioned on $\left\{ \boldsymbol{Y}(s) = \boldsymbol{y}(s), s \leq t \right\}$ for any $t \in [0, T]$.
\end{theorem}

\begin{proof}
    The proof is given in Appendix~\ref{sec:FMP_proof}.
\end{proof}

\begin{table}[ht]
    {\small
    \begin{tabular}{M{0.25\textwidth}|M{0.3\textwidth}|M{0.3\textwidth}|}
    \cline{2-3}
        & \textbf{\acf{MP}}, Theorem~\ref{th:MP} & \textbf{\acf{FMP}}, Theorem~\ref{th:FMP} \\ \hline
    \multicolumn{1}{|M{0.25\textwidth}|}{Main application} & Forward problem & Filtering problem \\ \hline
    \multicolumn{1}{|M{0.25\textwidth}|}{Preserves marginal distribution?}      & Yes \ding{52} & No \ding{56}            \\ \hline
    \multicolumn{1}{|M{0.25\textwidth}|}{Preserves marginal \textit{filtering}  distribution?}      & No \ding{56} & Yes \ding{52}            \\ \hline
    \multicolumn{1}{|M{0.25\textwidth}|}{Propensities can be estimated with}  & Samples from the unconditional distribution & Samples from the \textit{filtering} distribution \\ \hline
    \end{tabular}
    }
\caption{Comparison of the presented \ac{MP} theorems.}
\label{tab:MP_FMP_comparison}
\end{table}

Table~\ref{tab:MP_FMP_comparison} summarizes the differences between standard \ac{MP} and novel \ac{FMP} theorems. Next, we define the \ac{FMP} filtering problem as follows:
\begin{equation}
    \label{eq:pi_tilde_def}
    \tilde{\pi}'_{\boldsymbol{y}} (\boldsymbol{x}', t) := \Probcondmu{\tilde{\boldsymbol{X}}'(t) = \boldsymbol{x}' }{ \tilde{\boldsymbol{Y}}(s) = \boldsymbol{y}(s), s \leq t } .
\end{equation}
Theorem~\ref{th:FMP} guarantees that $\tilde{\pi}'_{\boldsymbol{y}} = \pi'_{\boldsymbol{y}}$. In other words, the solution of the filtering problem for the \ac{FMP} process $\tilde{\boldsymbol{Z}}$ is the same as \eqref{eq:filtering_problem_margin}. 

Note that in \eqref{eq:FMP_a_tilde_def} condition $\boldsymbol{Y}(t)=\boldsymbol{y}(t)$ intersects with condition $\boldsymbol{Z}'(t)=\boldsymbol{z}'$, therefore we can treat \ac{FMP} propensities $\Tilde{a}_j(\boldsymbol{z}', t) = \Tilde{a}_j(\boldsymbol{x}', \boldsymbol{y}, t)$ as functions of only $(\boldsymbol{x}', t)$. 

Ensuring the consistency of the estimation with the \ac{FMP} entails additional challenges compared to the \ac{UMP} filter. In the \ac{CMP} filter, we need to solve a filtering problem for the propensities in \eqref{eq:FMP_a_tilde_def}, as the \ac{FMP} requires observations $\{ \boldsymbol{y}(s), s \leq t \}$ to estimate the projected propensities $\tilde{a}_j(t)$. In this work, we use the \ac{PF} introduced in Section~\ref{subsubsec:PF} to estimate projected propensities:
\begin{equation}
\label{eq:a_tile_est}
    \Tilde{a}_j^M (\boldsymbol{x}', t) = \frac{\sum\limits_{k=1}^{M} w_k(t) 1_{\{ {\boldsymbol{V}}_k'(t) = {\boldsymbol{x}}' \}} a_j({\boldsymbol{V}}_k(t)) }{ \sum\limits_{k=1}^{M} w_k(t) 1_{\{ {\boldsymbol{V}}_k'(t) = {\boldsymbol{x}}' \}} },
\end{equation}
where $\left( \boldsymbol{V}_k, w_k \right)_{k=1}^{M}$ are particles for the full-dimensional process $\boldsymbol{Z}$. In practice, the denominator in \eqref{eq:a_tile_est} can be zero for some $(\boldsymbol{x}', t)$, therefore we use linear extrapolation on $\boldsymbol{x}'$ in such cases. For numerical stability, one can also reject estimates of \ac{PF} if the denominator is close to zero.

Algorithm~\ref{alg:FMP_FFSP} provides the general scheme of the \ac{CMP} filter based on the \ac{FMP} Theorem~\ref{th:FMP}.

\begin{algorithm}[ht]
\small
\caption{\ac{CMP} Filter}
\label{alg:FMP_FFSP}
\begin{algorithmic}[1]
    \REQUIRE Initial distribution $\pi(\cdot, 0)$ according to \eqref{eq:pi0_def}, observations: jump times $t_1, \dots, t_N$ and values $\boldsymbol{y}(t_1), \dots, \boldsymbol{y}(t_N)$, sample size $M$, truncated state space $\mathsf{X}'_N$ for $\Tilde{\boldsymbol{X}}'$
    \STATE Sample $\boldsymbol{V}_1(0), \dots, \boldsymbol{V}_M(0)$ according to $\pi(\cdot, 0)$, set $w_1(0) = \dots = w_M(0) = 1$
    \FOR{$k \in \{ 0 \dots n \}$}
        \STATE Simulate $\{ \boldsymbol{V}_i(t), w_i \}_{i=1}^{M}$ for $t \in [t_k, t_{k+1}]$ with the \ac{PF} (Section~\ref{subsubsec:PF}) for the full-dimensional \ac{SRN} $\boldsymbol{Z}$ 
        \STATE Resample $\{\boldsymbol{V}_i(t_{k+1})\}_{i=1}^{M}$ according to weights $\{w_i(t_{k+1})\}_{i=1}^{M}$, set all $w_i(t_{k+1}) := 1$.
        \STATE Estimate $\{\Tilde{a}_j (\cdot, t) \}_{j=1}^{J'}$ for $t \in [t_k, t_{k+1}]$ for the \ac{FMP}-\ac{SRN} $\Tilde{\boldsymbol{Z}}$ with the pointwise \ac{PF} estimates and linear in $\boldsymbol{x}'$ extrapolation
        \STATE Compute $\Tilde{\rho}_{\boldsymbol{y}}'(\cdot, t)$  for the \ac{FMP}-\ac{SRN} $\Tilde{\boldsymbol{Z}}$ for $t \in [t_k, t_{k+1})$ by applying \ac{FFSP} to \eqref{eq:filtering_equation_rho}
        \STATE Compute $\Tilde{\rho}_{\boldsymbol{y}}'(\cdot, t)$  for the \ac{FMP}-\ac{SRN} $\Tilde{\boldsymbol{Z}}$ for $t = t_{k+1}$ by applying \ac{FFSP} to \eqref{eq:filtering_equation_rho_jump}
        \STATE Compute $\Tilde{\pi}_{\boldsymbol{y}}'(\cdot, t)$ for $t \in [t_k, t_{k+1}]$ by normalizing  $\Tilde{\rho}_{\boldsymbol{y}}'(\cdot, t)$ according to  \cite{DAmbrosio2022FFSP}
    \ENDFOR
\end{algorithmic}
\end{algorithm}

The presented filter is a combination of two known filtering algorithms: the \ac{PF} and \ac{FFSP}. The \acf{CMP} filter  exploits the advantages of both. Instead of employing the \ac{PF} to estimate the conditional distribution directly (according to \eqref{eq:PF_weighed_avg}), it estimates the \ac{FMP} propensities. This replacement of the estimated function for the \ac{PF} could lower the variance and allow fewer particles to control the error compared to applying the \ac{PF} to the entire filtering problem, particularly when estimating rare events (e.g., the tails of the conditional distribution as shown in Section~\ref{subsec:linear_cascade}). For the \ac{FFSP}, the dimensionality of the state space is lowered, significantly reducing computational complexity compared to applying \ac{FFSP} to the full-dimensional filtering problem. Therefore, one can consider the presented \ac{CMP} filter as a variance reduction technique for the \ac{PF}. In this context, the propensities $\Tilde{a}_j$ are treated as auxiliary variables given by the expectations with additional conditioning on $\boldsymbol{Z}'(t)$, lowering the variance.

\begin{remark}
\label{remark:FMP_PF_refining}
    After applying the \ac{PF} in Algorithm~\ref{alg:FMP_FFSP}, one can address the original filtering problem with sampled particles $\{ \boldsymbol{V}_i(t), w_i \}_{i=1}^{M}$ and stop the computation if the obtained accuracy is satisfactory. In this sense, constructing $\tilde{\boldsymbol{Z}}'$ and applying \ac{FFSP} are refining steps for the \ac{PF}. However, our numerical experiments show that the additional \ac{FFSP} step is beneficial, especially when estimating the tails of the filtering distribution, as shown in Figure~\ref{fig:linear_cascade_log_distr}. 
\end{remark}

\subsection{Error Analysis}
\label{subsec:FMP_error_analysis}

This section provides an error analysis of the \ac{CMP} filter (Algorithm~\ref{alg:FMP_FFSP}). This section relabels $\pi'_{\boldsymbol{y}}$ as $\pi$ to simplify the notation because all \acp{PMF} used in this section are conditioned on the same trajectory $\boldsymbol{y}$ and marginalized to the species corresponding to $\boldsymbol{X}'$. Moreover, let $\{ \Tilde{a}_j^M \}_{j=1}^{J'}$ be the \ac{PF} estimator of the \ac{FMP} propensities $\{ \Tilde{a}_j \}_{j=1}^{J'}$ based on $M$ particles and $\Tilde{\boldsymbol{Z}}'^M$ be the approximation of the \ac{FMP}-\ac{SRN} obtained by replacing the propensities $\Tilde{a}_j$ with the estimates $\Tilde{a}_j^M$ \eqref{eq:a_tile_est}. 

Algorithm~\ref{alg:FMP_FFSP} returns an approximation of $\pi$ based on the following input parameters: the number of particles $M$, truncated state space $\mathsf{X}'_N$ for the \ac{FFSP}, and step size $\Delta t$ for the numerical \ac{ODE} solver. To investigate how these parameters affect the accuracy of the approximation, we introduce the following auxiliary filtering problems:
\begin{itemize}%[noitemsep, topsep=0pt]
    \item Let
        $\pi (\boldsymbol{x}', t) = \Probcondmu{\boldsymbol{X}'(t) = \boldsymbol{x}'}{ \boldsymbol{Y}(s) = \boldsymbol{y}(s), s \leq t }$ be the solution to the original marginal filtering problem \eqref{eq:filtering_problem_margin}.
    \item Let
        $\Tilde{\pi} (\boldsymbol{x}', t) = \Probcondmu{\Tilde{\boldsymbol{X}}'(t) = \boldsymbol{x}'}{ \Tilde{\boldsymbol{Y}}(s) = \boldsymbol{y}(s), s \leq t }$ be the solution to the filtering  problem for the process $\Tilde{\boldsymbol{Z}}' = \begin{bmatrix} \Tilde{\boldsymbol{X}}' \\  \Tilde{\boldsymbol{Y}} \end{bmatrix}$ .
    \item Let
        $\Tilde{\pi}^{M} (\boldsymbol{x}',t) =  \Probcondmu{\Tilde{\boldsymbol{X}}'^M(t) = \boldsymbol{x}'}{ \Tilde{\boldsymbol{Y}}^M(s) = \boldsymbol{y}(s), s \leq t }$ be the solution to the filtering  problem for the process $\Tilde{\boldsymbol{Z}}'^M = \begin{bmatrix} \Tilde{\boldsymbol{X}}'^M \\  \Tilde{\boldsymbol{Y}}^M \end{bmatrix}$ .
        \item Let 
        $\Tilde{\pi}^{M}_{FFSP} $ be the \ac{FFSP} approximation of $\Tilde{\pi}^{M}$ with the truncated state space $\mathsf{X}'_{N}$.
    \item Let 
        $\Tilde{\pi}^{M, \Delta t}_{FFSP} $ be the approximation of $\Tilde{\pi}^{M}_{FFSP}$ obtained as a numerical solution of the \ac{FFSP} system using discretization with the time step $\Delta t$.
\end{itemize}

Next, the total error is decomposed as follows:

\begin{align*}
    \abs{ \pi(\boldsymbol{x}',t) - \Tilde{\pi}^{M, \Delta t}_{FFSP}(\boldsymbol{x}',t)} &\leq \underbrace{\abs{\pi(\boldsymbol{x}',t) - \Tilde{\pi}(\boldsymbol{x}',t)}}_{\text{Model reduction error}} 
    + \underbrace{\abs{\Tilde{\pi}(\boldsymbol{x}',t) - \Tilde{\pi}^{M} (\boldsymbol{x}', t)}}_{\text{Projection error}} \\
    &+ \underbrace{\abs{\Tilde{\pi}^{M} (\boldsymbol{x}',t) - \Tilde{\pi}^{M}_{FFSP} (\boldsymbol{x}',t)}}_{\text{Truncation error}} \\
    &+ \underbrace{\abs{ \Tilde{\pi}^{M}_{FFSP} (\boldsymbol{x}',t) - \Tilde{\pi}^{M, \Delta t}_{FFSP} (\boldsymbol{x}',t)}}_{\text{\ac{ODE} solver error}} .
\end{align*}

According to Theorem~\ref{th:FMP}, the model reduction error $\abs{\pi(\boldsymbol{x}',t) - \Tilde{\pi}(\boldsymbol{x}',t)}$ is zero.

The truncation error can be controlled by including more states in the corresponding \ac{FFSP} system. The \ac{ODE} solver error depends on the selected numerical method and can be controlled by the time step $\Delta t$. These errors can be reduced to the desired tolerance without substantial computational cost because the dimensionality of the hidden space is low after projection.

The projection error $\abs{\Tilde{\pi}(\boldsymbol{x}',t) - \Tilde{\pi}^{M} (\boldsymbol{x}', t)}$ depends on the number of particles $M$ to approximate the \ac{FMP} propensities $\Tilde{a}_1, \dots, \Tilde{a}_{J'}$. Even for a fixed trajectory $\boldsymbol{y}([0,T])$,  $\Tilde{\pi}^M(\boldsymbol{x}', t)$ is a random variable because it depends on $M$ random particles. As $\Tilde{a}_j^M$ is a \ac{PF} estimator, it converges to $\Tilde{a}_j$ with a rate of $O(M^{-1/2})$. Some technical assumptions are necessary to demonstrate the same  order of convergence for $\Tilde{\pi}^M$.

Similarly to the assumption \eqref{eq:assumption_bound_propens} for the full-dimensional \ac{SRN}, we assume that all propensities of the \ac{FMP}-\ac{SRN} are bounded. That is, there exists a function $\Tilde{C}_2: \mathsf{Y} \to \mathbb{R}_{\geq 0}$ such that, for any $t \in [0, T]$, the following holds:
\begin{equation}
\label{eq:assumption_bound_propens_FMP} %\tag{A4}
    \sup_{\boldsymbol{x'} \in \mathsf{X}'} \sum_{j = 1}^{J'} \Tilde{a}_j \left( \boldsymbol{x}', \boldsymbol{y}, t \right) \leq C_2 (\boldsymbol{y}).
\end{equation}
All propensities should be bounded, not only observed ones as in \eqref{eq:assumption_bound_propens}.

The estimation of projection error reduces to a sensitivity analysis of the filtering problem to perturbations in the propensity functions. In the following, we state a general result for arbitrary $\varepsilon$-perturbed propensities $a^{\varepsilon}$ in Theorem~\ref{th:sensitivity} and then apply it for the \ac{CMP} filter with propensities estimated with \ac{PF} in Corollary~\ref{corollary:FMP_error_PF}.

\begin{theorem}[Sensitivity of the filtering problem for \acp{SRN}]
\label{th:sensitivity}
    Let $\Tilde{a}_j^{\varepsilon}(\boldsymbol{z}', t)$ be approximations of propensities $\Tilde{a}_j(\boldsymbol{z}', t)$, satisfying
    \begin{equation}
        \label{eq:a_tilde_error_eps}
        \Emu{\abs{ \Tilde{a}_j(\boldsymbol{z}', t) - \Tilde{a}_j^{\varepsilon}(\boldsymbol{z}', t) }} \leq \varepsilon, \quad j = 1, \dots, J'
    \end{equation}
    for all $\boldsymbol{z}' \in \mathsf{Z}'$ and $t \in [0, T]$. 
    Then, under the assumptions \eqref{eq:non_expl_cond} and \eqref{eq:assumption_bound_propens_FMP}, for all $t \in [0, T]$ 
    \begin{equation}
    \label{eq:projection_error_pi_eps}
        \Emu{ \sum_{\boldsymbol{x}' \in \mathsf{X}'} \abs{ \Tilde{\pi}(\boldsymbol{x}', t) - \Tilde{\pi}^{\varepsilon}(\boldsymbol{x}', t) }} =  O(\varepsilon),
    \end{equation}
    where $\Tilde{\pi}^{\varepsilon}(\boldsymbol{x}', t)$ is a solution of the filtering problem for the \ac{SRN} with propensities $\Tilde{a}_j^{\varepsilon}(\boldsymbol{z}', t)$.
\end{theorem}

\begin{proof}
    Appendix~\ref{sec:FMP_error_proof} provides the proof.
\end{proof}

As an immediate consequence of this theorem, the projection error rate of the \ac{CMP} filter under the assumption on the convergence rate of the \ac{PF}.

\begin{corollary}[Projection error in the \ac{MP} filter]
\label{corollary:FMP_error_PF}
     Let \eqref{eq:non_expl_cond} and \eqref{eq:assumption_bound_propens_FMP} hold. Suppose that 
     \begin{equation}
         \label{eq:a_tilde_error_M}
        \Emu{\abs{ \Tilde{a}_j(\boldsymbol{z}', t) - \Tilde{a}_j^{M}(\boldsymbol{z}', t) }} \leq O(M^{-1/2}), \quad j = 1, \dots, J'
     \end{equation}
    for all $\boldsymbol{z}' \in \mathsf{Z}'$ and $t \in [0, T]$. Then 
    \begin{equation}
    \label{eq:projection_error_pi}
        \Emu{ \sum_{\boldsymbol{x}' \in \mathsf{X}'} \abs{ \Tilde{\pi}(\boldsymbol{x}', t) - \Tilde{\pi}^M(\boldsymbol{x}', t) }} =  O(M^{-1/2}),
    \end{equation}
    for all $t \in [0, T]$.
\end{corollary}

The general form of Theorem~\ref{th:sensitivity} can be useful for other methods for estimating the propensities $\{\Tilde{a}_j\}_{j=1}^{J'}$. For instance, using \cite{leluc2025speeding, gerber2015sequential} could potentially lead to an error $O(M^{-p})$ with $p > 1/2$. Alternatively, one can use simpler filtering techniques (e.g., the Kalman filter) for propensities estimation and be sure that the bias of $\Tilde{\pi}^{\varepsilon}$ will not explode with respect to the bias of $\tilde{a}^{\varepsilon}_j$.

Although the order of error for the \ac{CMP} filter is the same as for the \ac{PF}, the constant in front of $M^{-1/2}$ is expected to be smaller for \ac{MP}. As we discussed earlier, the \ac{PF} is inefficient when estimating distributions due to the high variance of the indicator function. Whereas in the \ac{CMP} filter, the \ac{PF} estimates the projected propensities, yielding a smaller variance of the resulting estimator, as we show empirically in the next section (see Figures~\ref{fig:linear_cascade_log_distr}~and~\ref{fig:linear_cascade_convergence}).
\begin{remark}
    To the best of our knowledge, no results exist on the convergence of the specific version of \ac{PF} described in Section~\ref{subsubsec:PF}. However, for a general \ac{PF} with discrete observations, one can show the error decays as $O(M^{-1/2})$ if estimating a single quantity of interest \cite{Chopin2004CLTforSeqMC, crisan2002survey}.     
\end{remark}
The assumption \eqref{eq:a_tilde_error_M} might be too restrictive, as it requires infinitely many estimators ($\Tilde{a}^M(\boldsymbol{z}',t)$ for all states $\boldsymbol{z}'$ and times $t$) to be bounded by $O(M^{-1/2})$. In practice, we only need to estimate $\Tilde{a}(\boldsymbol{z}',t)$ for a finite number of states $\boldsymbol{z}'$ and discretized time points that will be included in the \ac{FFSP} equations.

    \section{Numerical Examples}
        \label{sec:4_Examples}
        This section presents two numerical examples of solving the marginal filtering problem \eqref{eq:filtering_problem_margin} for biochemical systems. The source code is available at \href{https://github.com/maksimchup/Markovian-Projection-in-filtering-for-SRNs}{\textit{github.com/maksimchup/Markovian-Projection-in-filtering-for-SRNs}}.

\subsection{Bistable Gene Expression Network}
\label{subsec:bistable_gene}

Consider an intracellular system with two genes \cite{Duso2018SelectedNodeSSA} sketched in Figure~\ref{fig:bistable_gene_diagram}. In an activated state, each gene can produce mRNA, which produces protein molecules. The amount of protein of each type affects the deactivation of the opposite gene. The model reactions are written as follows:
\begin{align*}
    \text{for } &i = 1,2 \quad i \neq j: \\[1.5pt]
    \text{mRNA}_i &\longleftrightarrow \emptyset \\
    \text{mRNA}_i &\longrightarrow \text{mRNA}_i + \text{protein}_i \\
    \text{protein}_i &\longrightarrow \emptyset \\
    \text{G}^\ast_j + \text{protein}_i &\longrightarrow \text{G}_j + \text{protein}_i  \\
    \text{G}^\ast_i &\longrightarrow \text{G}^\ast_i + \text{mRNA}_i \\
    \text{G}^\ast_i &\longleftrightarrow \text{G}_i,
\end{align*}
where $\text{G}^{\ast}$ and $\text{G}$ denote the activated and deactivated gene states.
For further numerical simulations, we use propensities according to the mass action kinetics \eqref{eq:mass_action} with the following reaction rates \cite{Duso2018SelectedNodeSSA}: $\theta_1 = \theta_2 = 0.1$, $\theta_3 = \theta_4 = 0.05$, $\theta_5 = \theta_6 = 5$, $\theta_7 = \theta_8 = 0.2$, $\theta_9 = \theta_{10} = 0.1$, $\theta_{11} = \theta_{12} = 1$, $\theta_{13} = \theta_{14} = 0.03$ and $\theta_{15} = \theta_{16} = 10^{-6}$.

\begin{figure}%[ht]
    \centering
    \includegraphics[width=0.7\linewidth]{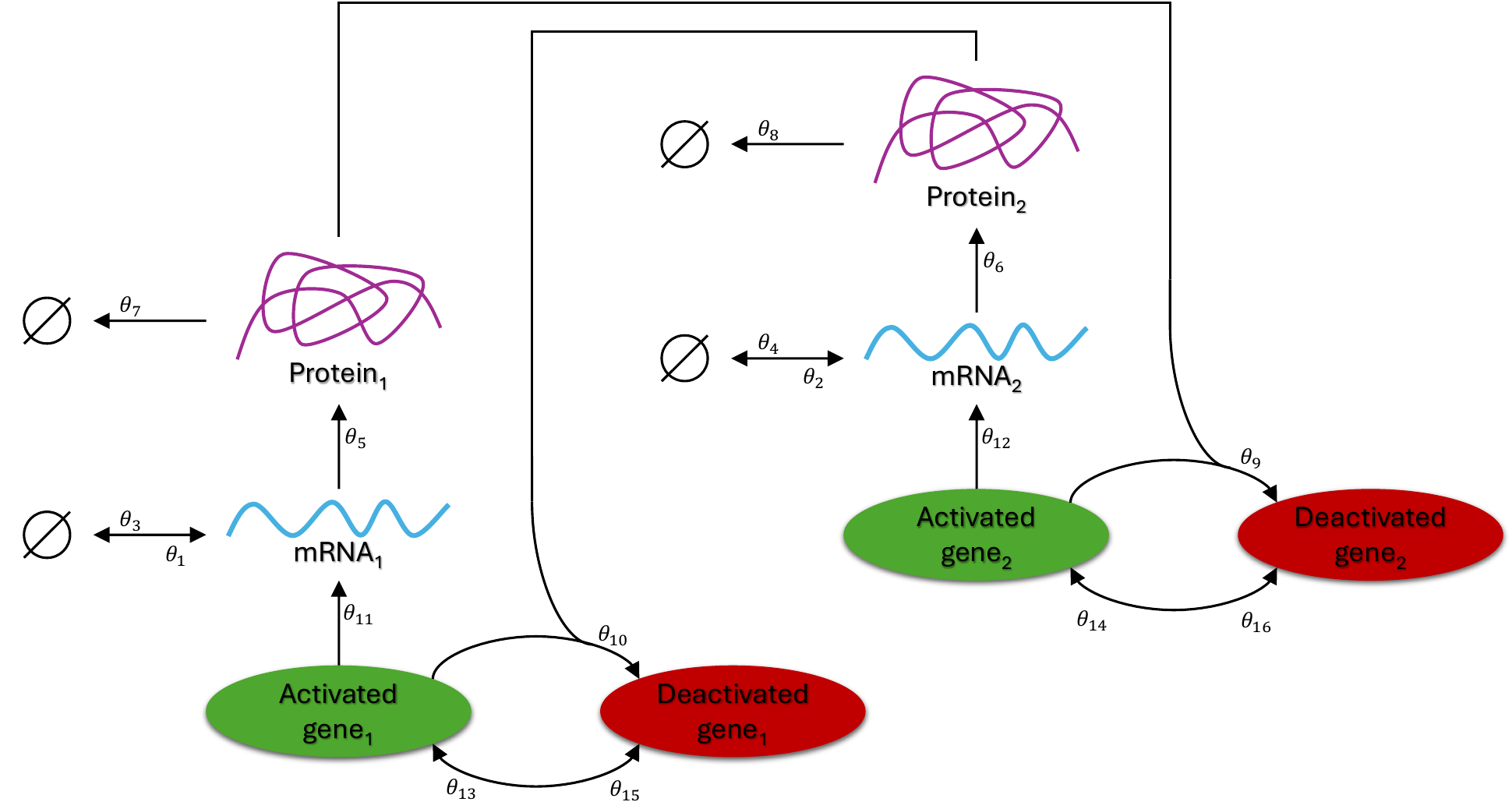}
    \caption{Reaction diagram of the bistable gene expression network (Section~\ref{subsec:bistable_gene}).}
    \label{fig:bistable_gene_diagram}
\end{figure}

Assume that the copy number of each protein is observed (i.e., the observed process $\boldsymbol{Y}$ is two-dimensional), and the goal is to estimate the conditional distribution of the amount of $\text{mRNA}_2$. We generate the observed part of the process using the \ac{SSA} (Figure~\ref{fig:bistable_gene_sol}(a)). Because we use synthetic data, the true trajectory of the hidden part is available for comparison with the corresponding conditional expectation from the solution of the filtering problem (see Figure~\ref{fig:bistable_gene_sol}(b)). However, the discrepancy in this case is explained not only by numerical error, but also by the stochastic nature of the problem itself. Determining the exact trajectory of the hidden part based on the information on the observed part is impossible.

\begin{figure}%[H]
    \centering
    \includegraphics[width=0.9\linewidth]{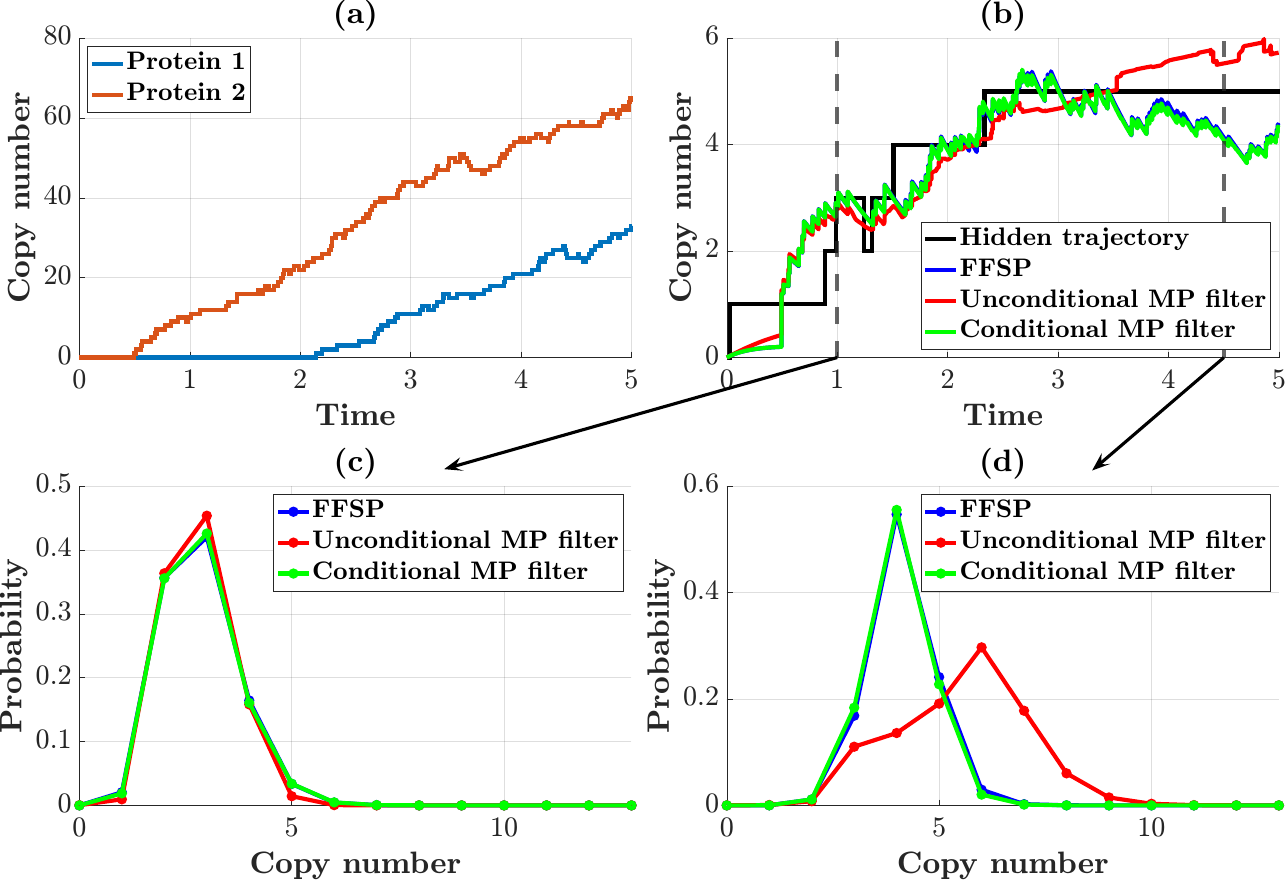}
    \caption{
    Numerical results for the bistable gene expression network (Section~\ref{subsec:bistable_gene}). The projection filters reduce the dimensionality of the hidden process from $6$ to $1$.
     (a): Observed trajectory of $\text{protein}_1$ and $\text{protein}_2$. 
     (b): Hidden trajectory of $\text{mRNA}_2$ and the corresponding estimates of its conditional expectation obtained with \ac{UMP} and \ac{CMP} filters and the \ac{FFSP} method for reference. 
     (c)--(d): Conditional distribution of $\text{mRNA}_2$ at time $t = 1$ and $t = 4.5$ obtained with the \ac{UMP} and \ac{CMP} filters and the \ac{FFSP} method for reference. 
     The \ac{UMP} has a larger error compared to the \ac{CMP} filter, which agrees with our theoretical results. 
     }
    \label{fig:bistable_gene_sol}
\end{figure}

The \ac{FFSP} solution for the full-dimensional system is utilized as a reference solution. In this model, gene states can only be 0 or 1, and the amounts of mRNA molecules are not bounded, so only mRNA$_1$ and mRNA$_2$ must be truncated when applying \ac{FFSP}. The upper bounds are set as $\text{mRNA}_1^{\text{max}} = \text{mRNA}_2^{\text{max}} = 30$, yielding $N = 2^4 \cdot 31^2 = 15 \, 376$ possible states for the full \ac{SRN}. 

For the proposed \ac{MP} filters, the projected process is three-dimensional (two observed  and one hidden species). The hidden space is still one-dimensional; thus, a significant efficiency improvement is expected compared to the full-dimensional process. With the same upper bound for mRNA, there are only $N' = 31$ hidden states. Thus, instead of $15 \, 376$ equations for the original \ac{SRN}, only $31$ should be solved for the projected \ac{SRN} (in both proposed projection filters).

\begin{table}%[ht]
    \centering 
    \def\arraystretch{1.25}
    {\small
    \begin{tabular}{|M{0.3\textwidth}|M{0.15\textwidth}|M{0.19\textwidth}|M{0.19\textwidth}|} 
    \cline{2-4}
        \multicolumn{1}{c|}{} & \textbf{Full model} & \textbf{\ac{UMP} filter} & \textbf{\ac{CMP} filter} \\ \hline
        {Dimensionality of hidden space} & 6 & 1 & 1 \\ \hline
        {Number of hidden states} & $15 \, 376$ & 31 & 31 \\ \hline
        {CPU time in seconds} & 841 & 5 & 5 \\ \hline
    \end{tabular}
    }
    \caption{The computational complexity of solving the filtering equation for the bistable gene network (Section~\ref{subsec:bistable_gene}) via the \ac{FFSP} method for the full model and proposed projection methods (\ac{UMP} and \ac{CMP} filters).}
    \label{tab:bistable_network_results}
\end{table}
 
 For the proposed \ac{MP} filters, we use sample size $M=10^3$ to estimate the projected propensities based on \eqref{eq:a_bar_est} and \eqref{eq:a_tile_est}.
 
The computational complexity of solving the filtering problem using the \ac{MP} filters compared to \ac{FFSP} is summarized in Table~\ref{tab:bistable_network_results}. The projection methods reduce the dimensionality of the hidden state space from $\dim{\boldsymbol{X}} = 6$ to $\dim{\boldsymbol{X}'} = 1$, resulting in a reduction in computational time from 841 seconds (s) for the full-dimensional \ac{SRN} to an average of 5 s for the \ac{UMP} and \ac{CMP} filters (i.e., an acceleration of about 160 times). Figure~\ref{fig:bistable_gene_sol} reveals that both \ac{MP} and \ac{FMP} surrogates provide a reasonable estimate for the conditional expectation. Despite the inconsistency of the \ac{UMP} filter, we obtain a result close to the reference solution up to time $t=3$. For large $t$, we see a significant error demonstrating the inconsistency of the \ac{UMP} filter. On the other hand, the \ac{CMP} filter yields estimates that are very close to the reference solution.

\subsection{Linear Cascade}
\label{subsec:linear_cascade}
Consider a linear cascade model \cite{Gupta2021DeepCME} consisting of $d$ species $S_1, \dots, S_d$. The reactions are given by
\begin{align*}
    \text{for } &i = 1,\dots, d: \\
    S_{i-1} &\longrightarrow S_{i}, \\
    S_i &\longrightarrow \emptyset,
\end{align*}
where $S_0 = \emptyset$. A sketch of this model is presented in Figure~\ref{fig:linear_cascade_diagram}. Further numerical simulations employ propensities according to the mass action kinetics \eqref{eq:mass_action} with the following reaction rates: $\theta_1 = 10$, $\theta_i = 5$ for $i=2,\dots, d$ and $\theta_{i}=1$ for $i=d,\dots, 2d$.

\begin{figure}%[ht]
    \centering
    \includegraphics[width=0.6\linewidth]{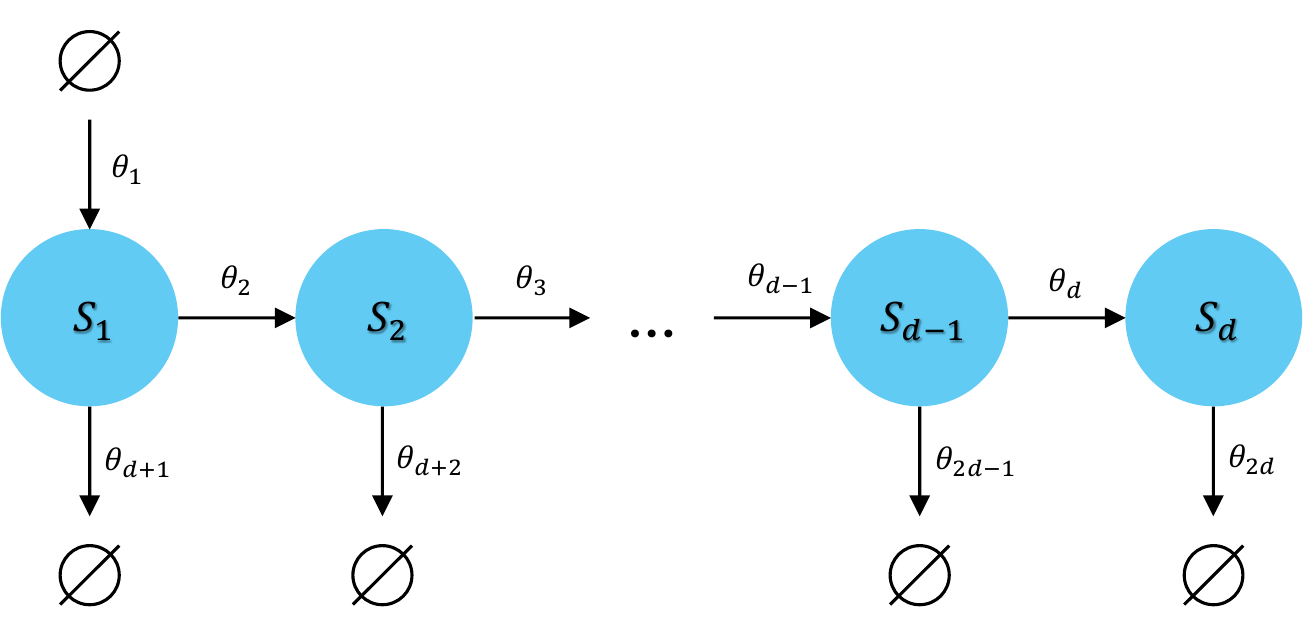}
    \caption{Reaction diagram of the bistable gene expression network (Section~\ref{subsec:linear_cascade}).}
    \label{fig:linear_cascade_diagram}
\end{figure}

Let us denote the copy number of $S_i$ at time $t$ by $Z_i(t)$ for $i = 1,\dots, d$ and consider an \ac{SRN} $\boldsymbol{Z}(t) = \left(Z_1(t), \dots, Z_d(t) \right)$. Assume that the $Z_d$ (copy number of $S_d$) is observed, and the goal is to estimate the conditional distribution of $Z_1$. As prior, the synthetic observed trajectory was simulated using the \ac{SSA}.

\begin{table}%[ht]
    \centering 
    \def\arraystretch{1.25}
    {\small 
    \begin{tabular}{|M{0.3\textwidth}|M{0.15\textwidth}|M{0.19\textwidth}|M{0.19\textwidth}|} 
    \cline{2-4}
        \multicolumn{1}{c|}{} & \textbf{Full model} & \textbf{\ac{UMP} filter} & \textbf{\ac{CMP} filter} \\ \hline
        {Dimensionality of hidden space} & 7 & 1 & 1 \\ \hline
        {Number of hidden states} & $> 1.9\times 10^7$ & 11 & 11 \\ \hline
        {CPU time in seconds} & $38 \, 166$ & 4.0 & 1.9 \\ \hline
    \end{tabular}
    }
    \caption{The computational complexity of solving the filtering problem for the linear cascade model (Section~\ref{subsec:linear_cascade}) with $d=8$ species via \ac{FFSP} method for the full model and \ac{MP} filters.}
    \label{tab:linear_cascade_results}
\end{table}

To obtain a reference solution, we used the \ac{FFSP} method with a $(d-1)$-dimensional truncated state space $\mathsf{X}_N = \{0, \dots, 10 \}^{(d-1)}$, resulting in a system of $N = 11^{(d-1)}$ equations. For the \ac{UMP} and \ac{CMP} filters, the projected hidden space is one-dimensional: $\mathsf{X}'_N = \{0, \dots, 10 \}$, which yields only  $N' = 11$ equations. For the \ac{UMP} and \ac{CMP} filters, we use sample size $M=500$ to estimate the projected propensities. The computational complexity of solving the filtering problem for $d=8$ using the \ac{MP} filters compared to \ac{FFSP} is summarized in Table~\ref{tab:linear_cascade_results}.

\begin{figure}%[ht]
    \centering                                 
    \includegraphics[width=0.9\linewidth]{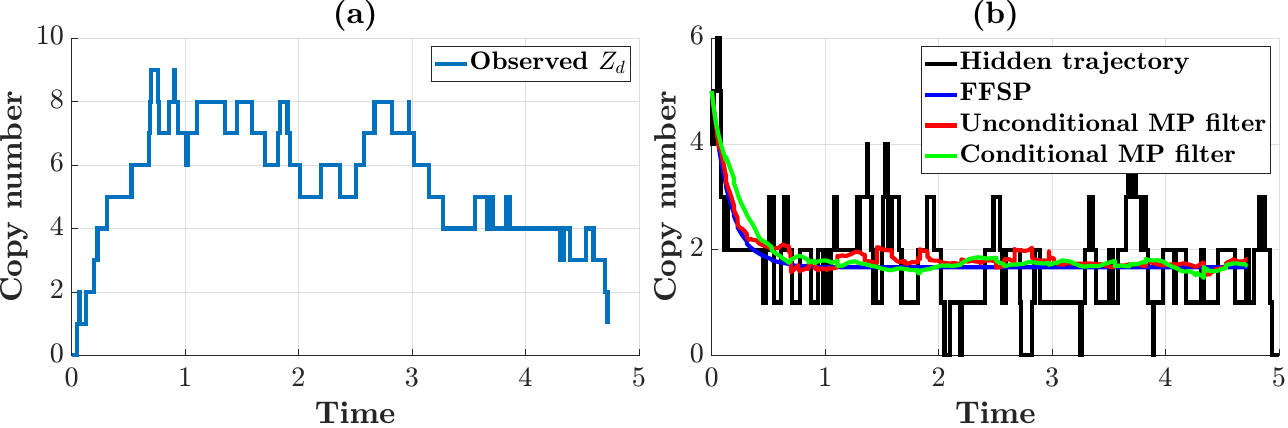}
    \caption{
    Numerical results for the linear cascade model (Section~\ref{subsec:linear_cascade}) with $d=5$ species. 
     (a): Observed trajectory of $Z_d$. 
     (b): Hidden trajectory of $Z_1$ and the corresponding estimates of its conditional expectation obtained with the \ac{UMP} filter, \ac{CMP} filter, and the \ac{FFSP} method for reference.
    }
    \label{fig:linear_cascade_means}
\end{figure}

The simulation results for $d=5$ are presented in Figure~\ref{fig:linear_cascade_means}. The reference solution shows that the estimated expectation is almost independent of the observations and rapidly reaches a nearly stationary state. This can be explained as follows: $S_1$ and $S_5$ are linked by reactions through three other species and therefore are almost independent. Due to the same reason, the additional conditioning on the observed trajectory $\{ \boldsymbol{Y}(s) = \boldsymbol{y}(s), s \leq t \} = \{ Z_d (s) = y(s), s \leq t \}$ should not significantly change the \ac{FMP} propensities \eqref{eq:FMP_a_tilde_def} from the \ac{UMP} propensities \eqref{eq:MP_a_bar_def}. At the beginning, the \ac{UMP} filter even outperforms the \ac{CMP} filter, but then it deviates more from the reference solution. Because there is no condition on $Z_d$ in the \ac{UMP} propensities, we have to extrapolate our estimates in the two-dimensional state space $(Z_1, Z_d)$, which can introduce larger errors compared to estimating \ac{FMP} propensities for which we perform extrapolation only for $Z_1$ (since $Z_d(s)$ can only be in state $y(s)$ according to \eqref{eq:FMP_a_tilde_def}).

\begin{figure}%[ht]
    \centering
    \includegraphics[width=0.5\linewidth]{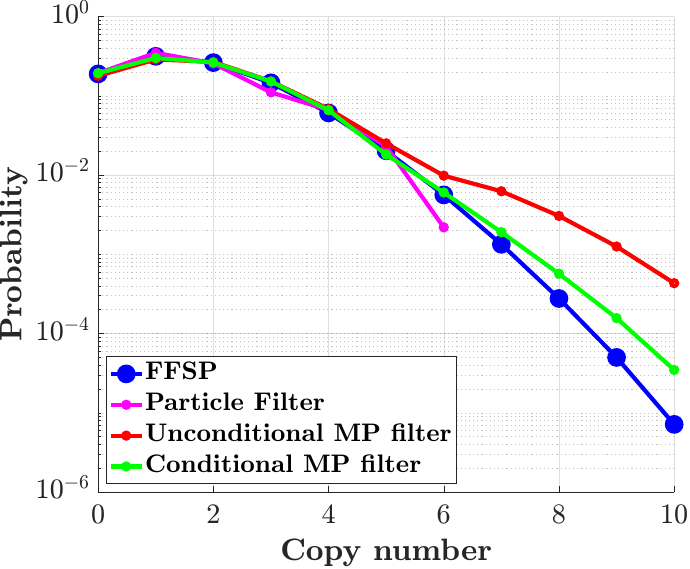}
    \caption{
    Numerical results for the $d=5$ dimensional linear cascade model (Section~\ref{subsec:linear_cascade}). Conditional \ac{PMF} of $Z_1(T)$ (in log scale), estimated with the \ac{PF}, \ac{UMP} and \ac{CMP} filters with sample size $M = 500$. The reference solution is obtained with the \ac{FFSP} method for the full model. For the \ac{PF}, no particles hit the region $\{ z_1 \geq 7 \}$, however, the \ac{CMP} filter based on the same set of particles led to a reasonable estimate of the tail probability.
    }
    \label{fig:linear_cascade_log_distr}
\end{figure}

Figure~\ref{fig:linear_cascade_log_distr} shows the difference between \ac{UMP} and \ac{CMP} filters in estimating the tails of the conditional distribution at the final time $T=5$. We also provide a \ac{PF} estimate, based on the same particles used to estimate the projected propensities in the \ac{CMP} filter. Clearly, the sample size $M=500$ is insufficient for the \ac{PF} to accurately estimate the probabilities in the tail, but applying the \ac{CMP} filter significantly improves the estimate.

For further comparison with the \ac{PF}, we consider the same five-dimensional system and the following \ac{QOI}: 
$$ 
    Q_d := \Probcondmu{Z_1(T) \geq 8}{Z_d(s) = y(s), s \leq T}.
$$ 

\begin{figure}%[ht]
    \centering
    \includegraphics[width=0.5\linewidth]{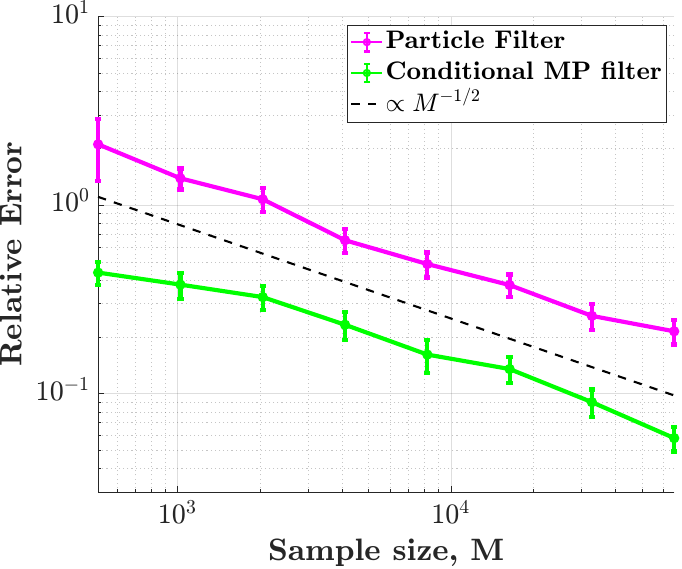}
    \caption{
    Numerical results for the linear cascade model (Section~\ref{subsec:linear_cascade}). 
    Expected relative errors in estimating $\Probcondmu{Z_1(T) \geq 8}{Z_d(s) = y(s), s \leq T}$ for $d=5$ with the \ac{PF} and \ac{CMP} filter, depending on the sample size (log-log scale). 
    Simulations were performed for fixed observed trajectories $Z_d$, and errors were averaged over 100 runs. The vertical bars show $95 \%$ confidence intervals.
    The results verify our convergence estimate for the \ac{CMP} filter (Corollary~\ref{corollary:FMP_error_PF}) and show that \ac{CMP} can be employed as a refining for the \ac{PF} (Remark~\ref{remark:FMP_PF_refining}). 
    }
    \label{fig:linear_cascade_convergence}
\end{figure}

The reference solution obtained with \ac{FFSP} is $Q^{\text{ref}}_{5} =  3.41 \times 10^{-4}$. Figure~\ref{fig:linear_cascade_convergence} presents the relative error of the \ac{PF}, \ac{UMP} and \ac{CMP} filters depending on the sample size $M$. The error of the \ac{CMP} filter is smaller than the error of \ac{PF}, confirming that the \ac{CMP} can be employed as an additional refining step for the \ac{PF} (see Remark~\ref{remark:FMP_PF_refining}). Moreover, Figure~\ref{fig:linear_cascade_convergence} shows that the convergence rate of the \ac{CMP} filter is $O(M^{-1/2})$, as derived in Section~\ref{subsec:FMP_error_analysis}. The \ac{UMP} filter has a significantly larger error and does not guarantee convergence.

\begin{figure}%[ht]
    \centering
    \includegraphics[width=0.5\linewidth]{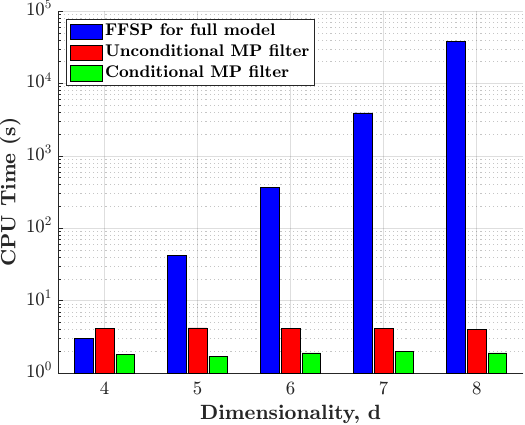}
    \caption{
     CPU times (log scale) of the \ac{UMP} filter, \ac{CMP} filter, and the \ac{FFSP} for reference, depending on the number of species in the linear cascade model (Section~\ref{subsec:linear_cascade}). 
     }
    \label{fig:linear_cascade_cpu_time}
\end{figure}

Figure~\ref{fig:linear_cascade_cpu_time} illustrates that the execution time of the \ac{FFSP} algorithm for the full-dimensional system increases exponentially as the dimensionality increases. In contrast, the time of the projection-based method does not change as the dimensionality increases. The execution time for the \ac{CMP} filter is less than for the \ac{UMP} filter because the \ac{PF} sampling involves only the hidden reactions from $\mathcal{U}$, whereas the \ac{MC} sampling involves all reactions.

    \section{Conclusions}
        \label{sec:conclusion}
        This work addressed the curse of dimensionality in the filtering problem for partially observable \acp{SRN}. Based on the \ac{FMP} theorem (Theorem~\ref{th:FMP}), we developed the \ac{CMP} filter to reduce the dimensionality (i.e., the number of species in the underlying \ac{SRN}). This approach is a modification of the standard \ac{MP} for the filtering problem. The \ac{CMP} filter is structurally identical to the unconditional one; the only difference is the additional conditioning on the observed trajectory in the expectation for the projected propensities. 

The proposed approach is to construct an \ac{SRN} with fewer species and solve the filtering problem for this network instead of the original one. This approach significantly reduces the dimensionality of the state space if the \ac{QOI} depends only on a small subset of hidden species. However, some propensities of this projected \ac{SRN} have no analytical expression and require numerical approximations. This work applies \ac{PF} to the original model to estimate the projected propensities. Using \ac{UMP} propensities estimated with the \ac{MC} methods is also possible but introduces additional errors. For the projected \ac{SRN}, we employed the \ac{FFSP} method to solve the filtering problem, demonstrating that the dimensionality reduction significantly increases its efficiency.

This work showed that applying \ac{CMP} and \ac{UMP} filters significantly reduces the computational complexity of the \ac{FFSP} method by reducing the dimensionality. In addition, the \ac{CMP} filter can be considered as a variance reduction for the \ac{PF} as we showed numerically.

The theoretical analysis demonstrated the consistency of the \ac{CMP} filter. The algorithm converges as $O(M^{-1/2})$, where $M$ is the number of particles to estimate the projected propensity functions. The numerical results confirmed the outperforming of the \ac{CMP} filter over the commonly used \ac{PF}. A more detailed analysis of the \ac{CMP} filter error, including relaxing assumptions and obtaining sharper error bounds, is an important direction for further work.

A possible direction for future work is applying \ac{FMP} (and standard \ac{MP}) to the \ac{PF}. Similarly to the \ac{FFSP} method, the \ac{PF} suffers from the curse of dimensionality due to weight degeneracy \cite{Snyder2008HighDimPF, Djuric2013HighDimPF}. The \ac{FMP} can significantly increase the efficiency of the \ac{PF}, but also requires an additional step for the propensity estimation, resulting in a two-step algorithm. The first step uses the \ac{PF} for the full model to estimate the projected propensities, and the second step employs \ac{PF} for the projected model to estimate the \ac{QOI}. The particles from the first step can also be applied for a rough estimation of the \ac{QOI}, which can be employed as a control variate. 

Another possibility for future work is to adapt \ac{FMP} for the filtering problem with noisy or discrete-time observations. In this case, the filtering equations have a different form but should also admit a linear equation for the unnormalized conditional \ac{PMF}. This linear equation allows applying the same techniques as that in the proof of Theorem~\ref{th:FMP}. Furthermore, it is also possible to incorporate parameter estimation into the filtering problem by including these parameters in the state vector, further increasing the dimensionality of the state space and making the \ac{FMP} approach even more relevant. 

Finally, one could extend the \ac{FMP} to the filtering problem for It\^o processes by deriving an equation for the marginalized conditional density by integrating both sides of the Zakai equation \cite{Zakai1969Optimal}. The idea is similar to the proof of Theorem~\ref{th:FMP} but may cause difficulties related to the continuity of the state space.

    \section*{Acknowledgments}
        This publication is based upon work supported by the King Abdullah University of Science and Technology (KAUST) Office of Sponsored Research (OSR). This work received funding from the Alexander von Humboldt Foundation through a Humboldt Professorship.

    \appendix
        \section{Unconditional Markovian Projection Filter}
\label{sec:MP_appendix}

In this appendix, we describe a filter based on a standard \ac{MP} Theorem~\ref{th:MP}. The first step is to compute the \ac{MP} propensities \eqref{eq:MP_a_bar_def}. Some of the propensities can be computed analytically, whereas others can be approximated using the \ac{MC} estimator:
\begin{equation}
\label{eq:a_bar_est}
    \bar{a}_j^M ({\boldsymbol{z}}', t) = \frac{\sum\limits_{i=1}^{M} 1_{\{ {\boldsymbol{Z}}_i'(t) = {\boldsymbol{z}}' \}} a_j({\boldsymbol{Z}}_i(t)) }{ \sum\limits_{i=1}^{M} 1_{\{ {\boldsymbol{Z}}_i'(t) = {\boldsymbol{z}}' \}} },
\end{equation}
where $\{ \boldsymbol{Z}_i \}_{i=1}^{M}$ are independent realizations of the full-dimensional process $\boldsymbol{Z}$. In practice, for some $({\boldsymbol{z}}', t)$, this estimate could be unreliable or singular due to the denominator being close to zero. In this work, we use a simple rule: if there were fewer than $10$ samples at point $(\boldsymbol{z}',t)$, then the estimate $\bar{a}_j(\boldsymbol{z}',t)$ is treated as unreliable, and the estimate \eqref{eq:a_bar_est} is replaced via linear extrapolation based on the states with reliable estimates.

The second step is solving the filtering equations for the \ac{MP}-\ac{SRN} of a lower dimensionality using the \ac{FFSP} method.

\begin{algorithm}[H]
\small
\caption{Unconditional \ac{MP} Filter}
\label{alg:MP_FFSP}
\begin{algorithmic}[1]
    \REQUIRE Initial distribution $\pi(\cdot, 0)$ according to \eqref{eq:pi0_def}, observations: jump times $t_1, \dots, t_N$ and values $\boldsymbol{y}(t_1), \dots, \boldsymbol{y}(t_N)$, sample size $M$, truncated state space $\mathsf{X}'_N$ for $\bar{\boldsymbol{X}}'$
    \STATE Sample $\boldsymbol{Z}_1(0), \dots, \boldsymbol{Z}_M(0)$ from $\mu$
    \FOR{$k \in \{ 0 \dots n \}$}
        \STATE Simulate $\{\boldsymbol{Z}_i(t)\}_{i=1}^M$ for  $t \in [t_k, t_{k+1}]$ from the full-dimensional \ac{SRN} $\boldsymbol{Z}$ using the \ac{SSA}
        \STATE Estimate $\{\bar{a}_j (\cdot, t)\}_{j=1}^{J'}$ for $t \in [t_k, t_{k+1}]$ using \eqref{eq:a_bar_est} and extrapolation for the \ac{MP}-\ac{SRN} $\bar{\boldsymbol{Z}}$ 
        \STATE Compute $\bar{\rho}_{\boldsymbol{y}}'(\cdot, t)$  for the \ac{MP}-\ac{SRN} $\bar{\boldsymbol{Z}}$ for $t \in [t_k, t_{k+1})$ by applying \ac{FFSP} to \eqref{eq:filtering_equation_rho}
        \STATE Compute $\bar{\rho}_{\boldsymbol{y}}'(\cdot, t)$  for the \ac{MP}-\ac{SRN} $\bar{\boldsymbol{Z}}$ for $t = t_{k+1}$ by applying \ac{FFSP} to \eqref{eq:filtering_equation_rho_jump}
        \STATE Compute $\bar{\pi}_{\boldsymbol{y}}'(\cdot, t)$ for $t \in [t_k, t_{k+1}]$ by normalizing  $\bar{\rho}_{\boldsymbol{y}}'(\cdot, t)$ according to  \cite{DAmbrosio2022FFSP}
    \ENDFOR
\end{algorithmic}
\end{algorithm}

Other methods can be used to estimate the projected propensities, e.g., the discrete $L^2$ regression \cite[Section~3.2]{Hammouda2023MP}, which is efficient if the shape of the projected propensity functions is known.

\section{Proof of Theorem~\ref{th:MP}}
\label{sec:MP_proof}
\begin{proof}
    The result can be derived from \cite[Theorem~3.1]{Hammouda2023MP}, but this work presents an alternative proof based on the marginalization of the \ac{CME}.
    The \ac{PMF} $p(\boldsymbol{z}, t) = \Probmu{\boldsymbol{Z}(t) = \boldsymbol{z}}$ obeys the following:
    \begin{small}
    \begin{equation}
    \label{eq:CME}
        \frac{\mathrm{d}}{\mathrm{d} t} p(\boldsymbol{z}, t) = \sum_{j=1}^{J} a_j(\boldsymbol{z}-\boldsymbol{\nu}_j) p(\boldsymbol{z}-\boldsymbol{\nu}_j, t) - \sum_{j=1}^{J} a_j(\boldsymbol{z}) p(\boldsymbol{z}, t)
    \end{equation}
    \end{small}
    with the initial condition $p(\cdot, 0)$, corresponding to the distribution $\mu$ of the random variable $\boldsymbol{Z}(0)$. The goal is to derive an equation for the probability function of the process $\boldsymbol{Z}'$: the marginal probability function $\boldsymbol{z}' \mapsto \sum\limits_{{\boldsymbol{z}}'' } p \left( \begin{bmatrix} \boldsymbol{z}' \\ {\boldsymbol{z}}'' \end{bmatrix}, t \right)$. To do so, we sum \eqref{eq:CME} over all states for $\boldsymbol{z}'' \in \mathbb{Z}^{\dim(\boldsymbol{Z}'')}$:
    \begin{small}
    \begin{align*}
        \sum_{{\boldsymbol{z}}'' } \frac{\mathrm{d}}{\mathrm{d} t} p \left( \begin{bmatrix} \boldsymbol{z}' \\ {\boldsymbol{z}}'' \end{bmatrix}, t \right) 
        &= \sum_{{\boldsymbol{z}}'' } \sum_{j=1}^{J} a_j \left( \begin{bmatrix} \boldsymbol{z}' - \boldsymbol{\nu}_j' \\ {\boldsymbol{z}}'' - \boldsymbol{\nu}_j'' \end{bmatrix} \right) p \left( \begin{bmatrix} \boldsymbol{z}' - \boldsymbol{\nu}_j' \\ {\boldsymbol{z}}'' - \boldsymbol{\nu}_j'' \end{bmatrix}, t \right) \\
        &- \sum_{{\boldsymbol{z}}'' } \sum_{j=1}^{J} a_j \left( \begin{bmatrix} \boldsymbol{z}' \\ {\boldsymbol{z}}'' \end{bmatrix} \right) p \left( \begin{bmatrix} \boldsymbol{z}' \\ {\boldsymbol{z}}'' \end{bmatrix}, t \right).
    \end{align*}
    \end{small}
    Under the non-explosivity assumption \eqref{eq:non_expl_cond}, the equation can be rewritten as follows:
    \begin{small}
    \begin{equation}
    \label{eq:MP_proof_1} 
    \begin{aligned}
         \frac{\mathrm{d}}{\mathrm{d} t} \left( \sum_{{\boldsymbol{z}}'' } p \left( \begin{bmatrix} \boldsymbol{z}' \\ {\boldsymbol{z}}'' \end{bmatrix}, t \right) \right)
         &= \sum_{j=1}^{J} \sum_{{\boldsymbol{z}}'' } a_j \left( \begin{bmatrix} \boldsymbol{z}' - \boldsymbol{\nu}_j' \\ {\boldsymbol{z}}'' - \boldsymbol{\nu}_j'' \end{bmatrix} \right) p \left( \begin{bmatrix} \boldsymbol{z}' - \boldsymbol{\nu}_j' \\ {\boldsymbol{z}}'' - \boldsymbol{\nu}_j'' \end{bmatrix}, t \right) \\
         &- \sum_{j=1}^{J} \sum_{{\boldsymbol{z}}'' } a_j \left( \begin{bmatrix} \boldsymbol{z}' \\ {\boldsymbol{z}}'' \end{bmatrix} \right) p \left( \begin{bmatrix} \boldsymbol{z}' \\ {\boldsymbol{z}}'' \end{bmatrix}, t \right)
    \end{aligned}
    \end{equation}
    \end{small}
    The left-hand side already has the desired marginal distribution. Consider the first sum on the right-hand side of \eqref{eq:MP_proof_1}:
    \begin{small}
    \begin{align*}
        \sum_{{\boldsymbol{z}}'' } a_j \left( \begin{bmatrix} \boldsymbol{z}' - \boldsymbol{\nu}_j' \\ {\boldsymbol{z}}'' - \boldsymbol{\nu}_j''\end{bmatrix} \right) &p \left( \begin{bmatrix}\boldsymbol{z}' - \boldsymbol{\nu}_j' \\ {\boldsymbol{z}}'' - \boldsymbol{\nu}_j'' \end{bmatrix}, t \right)  
        = \sum_{{\boldsymbol{z}}'' } a_j \left( \begin{bmatrix} \boldsymbol{z}' - \boldsymbol{\nu}_j' \\ {\boldsymbol{z}}'' \end{bmatrix} \right) p \left( \begin{bmatrix} \boldsymbol{z}' - \boldsymbol{\nu}_j' \\ {\boldsymbol{z}}'' \end{bmatrix}, t \right) \\
        &= \sum\limits_{{\boldsymbol{z}}'' } a_j \left( \begin{bmatrix} \boldsymbol{z}' - \boldsymbol{\nu}_j' \\ {\boldsymbol{z}}'' \end{bmatrix} \right)  \frac{ p \left( \begin{bmatrix} \boldsymbol{z}' - \boldsymbol{\nu}_j' \\ {\boldsymbol{z}}'' \end{bmatrix}, t \right) }{ \sum\limits_{{\boldsymbol{z}}'' } p \left( \begin{bmatrix}\boldsymbol{z}' - \boldsymbol{\nu}_j' \\ {\boldsymbol{z}}'' \end{bmatrix}, t \right)} \cdot \left( \sum_{{\boldsymbol{z}}'' } p \left( \begin{bmatrix} \boldsymbol{z}' - \boldsymbol{\nu}_j' \\ {\boldsymbol{z}}'' \end{bmatrix}, t \right)  \right) \\
        &= \underbrace{ \Econdmu{a_j \left( \begin{bmatrix} \boldsymbol{Z}'(t) \\ \boldsymbol{Z}''(t) \end{bmatrix} \right)}{\boldsymbol{Z}'(t) = \boldsymbol{z}' - \boldsymbol{\nu}_j' } }_{\textstyle = \bar{a}_j (\boldsymbol{z}' - \boldsymbol{\nu}_j', t)} \cdot \left( \sum_{{\boldsymbol{z}}'' } p \left( \begin{bmatrix} \boldsymbol{z}' - \boldsymbol{\nu}_j' \\ {\boldsymbol{z}}'' \end{bmatrix}, t \right) \right).
    \end{align*}
    \end{small}
    The denominator is zero only if the whole expression is zero; in this case, $(\boldsymbol{z}'-\boldsymbol{\nu}_j', \cdot)$ can be excluded from the state space because it is unreachable.
    
    Similarly, the second sum on the right-hand side of \eqref{eq:MP_proof_1} is transformed:
    \begin{small}
    \begin{align*}
        \sum_{{\boldsymbol{z}}'' } a_j \left( \begin{bmatrix} \boldsymbol{z}' \\ {\boldsymbol{z}}'' \end{bmatrix} \right) p \left( \begin{bmatrix} \boldsymbol{z}' \\ {\boldsymbol{z}}'' \end{bmatrix} , t \right) 
        %&= \frac{\sum\limits_{{\boldsymbol{z}}'' } a_j\left( \begin{bmatrix} \boldsymbol{z}' \\ {\boldsymbol{z}}'' \end{bmatrix} \right) p\left( \begin{bmatrix} \boldsymbol{z}' \\ {\boldsymbol{z}}'' \end{bmatrix}, t \right)}{\sum\limits_{{\boldsymbol{z}}'' } p\left( \begin{bmatrix} \boldsymbol{z}' \\ {\boldsymbol{z}}'' \end{bmatrix}, t \right)} \left( \sum_{{\boldsymbol{z}}'' } p\left( \begin{bmatrix} \boldsymbol{z}' \\ {\boldsymbol{z}}'' \end{bmatrix}, t \right) \right) \\
        &=  \underbrace{ \Econdmu{a_j\left( \begin{bmatrix} \boldsymbol{z}' \\ {\boldsymbol{z}}'' \end{bmatrix} \right)}{\boldsymbol{Z}'(t) = \boldsymbol{z}' } }_{\textstyle = \bar{a}_j (\boldsymbol{z}', t)} \left( \sum_{{\boldsymbol{z}}'' } p\left( \begin{bmatrix} \boldsymbol{z}' \\ {\boldsymbol{z}}'' \end{bmatrix}, t \right) \right).
    \end{align*}
    \end{small}
    The denominator here is also not zero due to the same reason.
    
    The results reveal that \eqref{eq:MP_proof_1} can be written as follows:
    \begin{small}
    \begin{align*}
        \frac{\mathrm{d}}{\mathrm{d} t} \left( \sum_{{\boldsymbol{z}}'' } p\left( \begin{bmatrix} \boldsymbol{z}' \\ {\boldsymbol{z}}'' \end{bmatrix}, t \right) \right) 
        &= \sum_{j=1}^{J} \bar{a}_j (\boldsymbol{z}' - \boldsymbol{\nu}_j', t) \left( \sum_{{\boldsymbol{z}}'' } p\left( \begin{bmatrix} \boldsymbol{z}' - \boldsymbol{\nu}_j' \\ {\boldsymbol{z}}'' \end{bmatrix}, t \right) \right) \\
        &- \sum_{j=1}^{J} \bar{a}_j (\boldsymbol{z}', t) \left( \sum_{{\boldsymbol{z}}'' } p\left( \begin{bmatrix} \boldsymbol{z}' \\ {\boldsymbol{z}}'' \end{bmatrix}, t \right) \right) .
    \end{align*}
    \end{small}
    Thus, we obtained the \ac{ODE} for the probability function of the process $\boldsymbol{Z}'$, which is the same as the \ac{CME} for the process $\bar{\boldsymbol{Z}}'$. Furthermore, the initial conditions for these \acp{ODE} coincide because $\bar{\boldsymbol{Z}}'(0) \overset{d}{=} \boldsymbol{Z}'(0)$. Finally, the statement of the theorem follows from the uniqueness of the solution to the initial value problem.
\end{proof}

\section{Proof of Theorem~\ref{th:FMP}}
\label{sec:FMP_proof}

\begin{proof}  
    According to the filtering equation \eqref{eq:filtering_equation_rho}, the unnormalized conditional probability function $\rho_{\boldsymbol{y}}(\boldsymbol{x}, t)$ of $\boldsymbol{X}(t)$ for $t \in (t_k, t_{k+1})$ satisfies the following:
    \begin{small}
    \begin{align*}
        \frac{\mathrm{d}}{\mathrm{d} t} \rho_{\boldsymbol{y}}(\boldsymbol{x}, t) 
        &= \sum_{j \in \mathcal{U}} \rho_{\boldsymbol{y}}(\boldsymbol{x} - \boldsymbol{\nu}_{\boldsymbol{x}, j}, t) a_j \left( \boldsymbol{x}-\boldsymbol{\nu}_{\boldsymbol{x}, j}, \boldsymbol{y}(t_k) \right) - \sum_{j=1}^{J} \rho_{\boldsymbol{y}}(\boldsymbol{x}, t) a_j \left( \boldsymbol{x}, \boldsymbol{y}(t_k) \right).
    \end{align*}
    \end{small}
    Summing these equations for all possible states for $\boldsymbol{X}''(t) \in \mathbb{Z}^{d''}$ yields
    \begin{small}
    \begin{align*}
        \sum_{\boldsymbol{x}'' \in \mathbb{Z}^{d''}} \frac{\mathrm{d}}{\mathrm{d} t} \rho_{\boldsymbol{y}} \left( \begin{bmatrix} \boldsymbol{x}' \\ \boldsymbol{x}'' \end{bmatrix} , t \right) 
        &= \sum_{\boldsymbol{x}'' \in \mathbb{Z}^{d''}}\sum_{j \in \mathcal{U}} a_j \left( \begin{bmatrix} \boldsymbol{x}' - \boldsymbol{\nu}_{\boldsymbol{x}, j}' \\ \boldsymbol{x}'' - \boldsymbol{\nu}_{\boldsymbol{x}, j}'' \end{bmatrix}, \boldsymbol{y}(t_k) \right) \rho_{\boldsymbol{y}} \left( \begin{bmatrix} \boldsymbol{x}' - \boldsymbol{\nu}_{\boldsymbol{x}, j}' \\ \boldsymbol{x}'' - \boldsymbol{\nu}_{\boldsymbol{x}, j}''\end{bmatrix} , t \right)  \\
        &- \sum_{\boldsymbol{x}'' \in \mathbb{Z}^{d''}} \sum_{j=1}^{J} a_j \left( \begin{bmatrix} \boldsymbol{x}' \\ \boldsymbol{x}'' \end{bmatrix}, \boldsymbol{y}(t_k) \right) \rho_{\boldsymbol{y}} \left( \begin{bmatrix} \boldsymbol{x}' \\ \boldsymbol{x}'' \end{bmatrix}, t \right) .
    \end{align*}
    \end{small}
    Under the non-explosivity assumption \eqref{eq:non_expl_cond}, it can be rewritten as
    \begin{small}
    \begin{equation}
    \label{eq:FMP_proof_1} 
    \begin{aligned}
        \frac{\mathrm{d}}{\mathrm{d} t} \left( \sum_{\boldsymbol{x}'' \in \mathbb{Z}^{d''}} \rho_{\boldsymbol{y}} \left( \begin{bmatrix} \boldsymbol{x}' \\ \boldsymbol{x}'' \end{bmatrix}, t \right) \right) 
        &= \sum_{j \in \mathcal{U}} \sum_{\boldsymbol{x}'' \in \mathbb{Z}^{d''}} a_j \left( \begin{bmatrix} \boldsymbol{x}' - \boldsymbol{\nu}_{\boldsymbol{x}, j}' \\ \boldsymbol{x}'' - \boldsymbol{\nu}_{\boldsymbol{x}, j}'' \end{bmatrix}, \boldsymbol{y}(t_k) \right) \rho_{\boldsymbol{y}} \left( \begin{bmatrix} \boldsymbol{x}' - \boldsymbol{\nu}_{\boldsymbol{x}, j}' \\ \boldsymbol{x}'' - \boldsymbol{\nu}_{\boldsymbol{x}, j}''\end{bmatrix}, t \right)  \\
        &- \sum_{j=1}^{J} \sum_{\boldsymbol{x}'' \in \mathbb{Z}^{d''}} a_j \left( \begin{bmatrix} \boldsymbol{x}' \\ \boldsymbol{x}'' \end{bmatrix}, \boldsymbol{y}(t_k) \right) \rho_{\boldsymbol{y}} \left( \begin{bmatrix} \boldsymbol{x}' \\ \boldsymbol{x}'' \end{bmatrix}, t \right) .
    \end{aligned}
    \end{equation}
    \end{small}

    Consider the first sum in \eqref{eq:FMP_proof_1}:
    \begin{footnotesize}
    \begin{align*}
        &\sum_{\boldsymbol{x}'' \in \mathbb{Z}^{d''}} a_j \left( \begin{bmatrix} \boldsymbol{x}' - \boldsymbol{\nu}_{\boldsymbol{x}, j}' \\ \boldsymbol{x}'' - \boldsymbol{\nu}_{\boldsymbol{x}, j}'' \end{bmatrix}, \boldsymbol{y}(t_k) \right) \rho_{\boldsymbol{y}} \left( \begin{bmatrix} \boldsymbol{x}' - \boldsymbol{\nu}_{\boldsymbol{x}, j}' \\ \boldsymbol{x}'' - \boldsymbol{\nu}_{\boldsymbol{x}, j}'' \end{bmatrix}, t \right) \\
        &= \sum_{\boldsymbol{x}'' \in \mathbb{Z}^{d''}} a_j \left( \begin{bmatrix} \boldsymbol{x}' - \boldsymbol{\nu}_{\boldsymbol{x}, j}' \\ \boldsymbol{x}'' \end{bmatrix}, \boldsymbol{y}(t_k) \right) \rho_{\boldsymbol{y}} \left( \begin{bmatrix} \boldsymbol{x}' - \boldsymbol{\nu}_{\boldsymbol{x}, j}' \\ \boldsymbol{x}'' \end{bmatrix}, t \right) \\
        &= \sum_{\boldsymbol{x}'' \in \mathbb{Z}^{d''}} a_j \left( \begin{bmatrix} \boldsymbol{x}' - \boldsymbol{\nu}_{\boldsymbol{x}, j}' \\ \boldsymbol{x}'' \end{bmatrix}, \boldsymbol{y}(t_k) \right) \frac{ \rho_{\boldsymbol{y}} \left( \begin{bmatrix} \boldsymbol{x}' - \boldsymbol{\nu}_{\boldsymbol{x}, j}' \\ \boldsymbol{x}'' \end{bmatrix}, t \right) }{ \sum\limits_{\boldsymbol{x}'' \in \mathbb{Z}^{d''}} \rho_{\boldsymbol{y}}\left( \begin{bmatrix} \boldsymbol{x}' - \boldsymbol{\nu}_{\boldsymbol{x}, j}' \\ \boldsymbol{x}'' \end{bmatrix}, t \right) }  \left( \sum_{\boldsymbol{x}'' \in \mathbb{Z}^{d''}} \rho_{\boldsymbol{y}} \left( \begin{bmatrix} \boldsymbol{x}' - \boldsymbol{\nu}_{\boldsymbol{x}, j}' \\ \boldsymbol{x}'' \end{bmatrix}, t \right) \right) \\
        &= \sum_{\boldsymbol{x}'' \in \mathbb{Z}^{d''}} a_j \left( \begin{bmatrix} \boldsymbol{x}' - \boldsymbol{\nu}_{\boldsymbol{x}, j}' \\ \boldsymbol{x}'' \end{bmatrix}, \boldsymbol{y}(t_k) \right) \frac{ \pi_{\boldsymbol{y}} \left( \begin{bmatrix} \boldsymbol{x}' - \boldsymbol{\nu}_{\boldsymbol{x}, j}' \\ \boldsymbol{x}'' \end{bmatrix}, t \right) }{ \sum\limits_{\boldsymbol{x}'' \in \mathbb{Z}^{d''}} \pi_{\boldsymbol{y}} \left( \begin{bmatrix} \boldsymbol{x}' - \boldsymbol{\nu}_{\boldsymbol{x}, j}' \\ \boldsymbol{x}'' \end{bmatrix}, t \right) }  \left( \sum_{\boldsymbol{x}'' \in \mathbb{Z}^{d''}} \rho_{\boldsymbol{y}} \left( \begin{bmatrix} \boldsymbol{x}' - \boldsymbol{\nu}_{\boldsymbol{x}, j}' \\ \boldsymbol{x}'' \end{bmatrix}, t \right) \right) .
    \end{align*}
    \end{footnotesize}
    The last transition follows from dividing the numerator and denominator by the normalization factor $\sum\limits_{\boldsymbol{x}' \in \mathbb{Z}^{d'}} \sum\limits_{\boldsymbol{x}'' \in \mathbb{Z}^{d''}} \rho_{\boldsymbol{y}} \left( \begin{bmatrix} \boldsymbol{x}' \\ \boldsymbol{x}'' \end{bmatrix}, t \right)$. The denominator is zero only if all hidden states $\begin{bmatrix} \boldsymbol{x}' \\ \cdot \end{bmatrix}$ are unreachable and can be excluded.
    % Moreover,
    % \begin{small}
    % \begin{align*}
    %     \frac{ \pi_{\boldsymbol{y}} \left( \begin{bmatrix} \boldsymbol{x}' - \boldsymbol{\nu}_{\boldsymbol{x}, j}' \\ \boldsymbol{x}'' \end{bmatrix}, t \right) }{ \sum\limits_{\boldsymbol{x}'' \in \mathbb{Z}^{d''}} \pi_{\boldsymbol{y}} \left( \begin{bmatrix} \boldsymbol{x}' - \boldsymbol{\nu}_{\boldsymbol{x}, j}' \\ \boldsymbol{x}'' \end{bmatrix}, t \right) } = \Probcondmu{ \boldsymbol{X}''(t) = \boldsymbol{x}''}{ \boldsymbol{X}'(t) =  \boldsymbol{x}' - \boldsymbol{\nu}_{\boldsymbol{x}, j}'  , \boldsymbol{Y}(s) = \boldsymbol{y}(s) , s \leq t }.
    % \end{align*}
    % \end{small}
    Therefore, the first term on the right-hand side of \eqref{eq:FMP_proof_1} simplifies to 
    \begin{small}
    \begin{align*}
        \sum_{j \in \mathcal{U}} \underbrace{\Econdmu{a_j (\boldsymbol{Z}(t))}{ \boldsymbol{X}'(t) = \boldsymbol{x}'-\boldsymbol{\nu}_{\boldsymbol{x}, j}' , \boldsymbol{Y}(s) = \boldsymbol{y}(s), s \leq t}}_{\textstyle = \Tilde{a}_j (\boldsymbol{x}'-\boldsymbol{\nu}_{\boldsymbol{x}, j}', \boldsymbol{y}(t_k), t) } \left( \sum_{\boldsymbol{x}'' \in \mathbb{Z}^{d''}} \rho_{\boldsymbol{y}} \left( \begin{bmatrix} \boldsymbol{x}' \\ \boldsymbol{x}'' \end{bmatrix}, t \right) \right).
    \end{align*}
    \end{small}
    
    The same procedure reveals that the second term on the right-hand side of \eqref{eq:FMP_proof_1} is equal to
    \begin{small}
    \begin{align*}
        \sum_{j = 1}^{J} \underbrace{\Econdmu{a_j (\boldsymbol{Z}(t))}{ \boldsymbol{X}'(t) = \boldsymbol{x}', \boldsymbol{Y}(s) = \boldsymbol{y}(s), s \leq t}}_{\textstyle = \Tilde{a}_j (\boldsymbol{x}', \boldsymbol{y}(t_k), t) } \cdot \left( \sum_{\boldsymbol{x}'' \in \mathbb{Z}^{d''}} \rho_{\boldsymbol{y}} \left( \begin{bmatrix} \boldsymbol{x}' \\ \boldsymbol{x}'' \end{bmatrix}, t \right) \right).
    \end{align*}
    \end{small}
    Gathering all the results into \eqref{eq:FMP_proof_1} yields the following:
    \begin{small}
    \begin{align*}
        \frac{\mathrm{d}}{\mathrm{d} t} \Tilde{\rho}(\boldsymbol{x}', t)
        &= \sum_{j \in \mathcal{U}} \Tilde{a}_j (\boldsymbol{x}' - \boldsymbol{\nu}_{\boldsymbol{x}, j}', \boldsymbol{y}(t_k), t) \Tilde{\rho}(\boldsymbol{x}' - \boldsymbol{\nu}_{\boldsymbol{x}, j}', t) - \sum_{j=1}^{J} \Tilde{a}_j (\boldsymbol{x}',  \boldsymbol{y}(t_k)) \Tilde{\rho}(\boldsymbol{x}', t) ,
    \end{align*}
    \end{small}
    where  $\Tilde{\rho}(\boldsymbol{x}', t) = \sum\limits_{\boldsymbol{x}'' \in \mathbb{Z}^{d''}} \rho_{\boldsymbol{y}} \left( \begin{bmatrix} \boldsymbol{x}' \\ \boldsymbol{x}'' \end{bmatrix}, t \right) $.
    This is exactly the filtering equation for $\Tilde{\boldsymbol{Z}}'$ between jumps. 
    
    Using the same summation procedure, also proves that the updating according to \eqref{eq:filtering_equation_pi_jump} at the jump times for the marginal distribution of $\boldsymbol{X}'$ is the same as that for $\Tilde{\boldsymbol{X}}'$. Finally, $\Tilde{\boldsymbol{Z}}'(0) \overset{d}{=} \boldsymbol{Z}'(0)$ ensures the same initial conditions for the filtering equation; hence, the solutions coincide.
\end{proof}

\section{Proof of Theorem~\ref{th:sensitivity}}
\label{sec:FMP_error_proof}
\begin{proof}

Consider $\Tilde{\pi}(\boldsymbol{x}',t)$ and $\Tilde{\pi}^{\varepsilon} (\boldsymbol{x}', t)$ as solutions to the corresponding filtering equations. Let $\mathsf{X}' = \{ \boldsymbol{x}'_1, \boldsymbol{x}'_2, \dots \}$ be the infinite state space and $\Tilde{\pi}(t) = (\Tilde{\pi}(\boldsymbol{x}'_1, t), \Tilde{\pi}(\boldsymbol{x}'_2, t), \dots )^\top$  be the infinite-dimensional vector, then the corresponding filtering equation \eqref{eq:filtering_equation_pi} can be written as follows:
\begin{small}
\begin{equation}
    \label{eq:filtering_pi_matrix}
    \frac{\mathrm{d}}{\mathrm{d} t} \Tilde{\pi}(t) = \mathcal{A}(t) \Tilde{\pi}(t) + \langle \alpha(t), \Tilde{\pi}(t) \rangle \cdot \Tilde{\pi}(t), \qquad t\in(t_k, t_{k+1}),
\end{equation}
\end{small}
where $\mathcal{A}(t)$ is a linear operator given by the infinite-dimensional matrix with the entries
\begin{small}
\begin{align*}
    \mathcal{A}_{n m}(t) := 
    \begin{cases}
        - \sum\limits_{j = 1}^{J'} \Tilde{a}_j(\boldsymbol{x}_{n}, \boldsymbol{y}(t), t) & \text{if } \boldsymbol{x}_{n} = \boldsymbol{x}_{m} \\
        \Tilde{a}_j(\boldsymbol{x}_{m}, \boldsymbol{y}(t), t) & \text{if } \exists j \in \mathcal{U}': \boldsymbol{x}_{n}' = \boldsymbol{x}_{m}' + \boldsymbol{\nu}_{\boldsymbol{x}, j}' \\
        0 & \text{otherwise}
    \end{cases}
\end{align*}
\end{small}
and $\langle \alpha(t), \cdot \rangle$ is a linear functional given by a scalar product with the infinite-dimensional vector $\alpha(t)$ with entries
$
% \begin{small}
% \begin{align*}
    \alpha_n(t) := \sum_{j \in \mathcal{O}} \Tilde{a}_j (\boldsymbol{x}_n', \boldsymbol{y}(t), t).
% \end{align*}
% \end{small}
$
Equation \eqref{eq:filtering_equation_pi_jump} can be written as follows:
\begin{small}
\begin{equation}
    \label{eq:filtering_pi_jump_matrix}
    \Tilde{\pi}(t_k) = \frac{\mathcal{B}(t_k^-) \Tilde{\pi}(t_k^-)}{\norm{\mathcal{B}(t_k^-) \Tilde{\pi}(t_k^-)}_1},
\end{equation}
\end{small}
where $\mathcal{B}(t_k^-)$ is a linear operator given by the infinite-dimensional matrix with entries
\begin{small}
\begin{align*}
    \mathcal{B}_{n m}(t_k^-) := 
    \begin{cases}
        \frac{1}{\abs{\mathcal{O}_k}} \Tilde{a}_j(\boldsymbol{x}_{m}', \boldsymbol{y}(t_k^-), t_k^-) & \text{if } \exists j \in \mathcal{O}_k: \boldsymbol{x}_{n}' = \boldsymbol{x}_{m}' + \boldsymbol{\nu}_{\boldsymbol{x}, j}' \\
        0 & \text{otherwise}.
    \end{cases}
\end{align*}
\end{small}
We denote by $\mathcal{A}^{\varepsilon}$, $\alpha^{\varepsilon}$, and $\mathcal{B}^{\varepsilon}$ the corresponding operators obtained by replacing all propensities $\Tilde{a}$ with their its estimates $\Tilde{a}^{\varepsilon}$. Therefore, $\Tilde{\pi}^{\varepsilon}$ obeys
\begin{small}
\begin{equation}
    \label{eq:filtering_piM_matrix}
    \frac{\mathrm{d}}{\mathrm{d} t} \Tilde{\pi}^{\varepsilon}(t) = \mathcal{A}^{\varepsilon}(t) \Tilde{\pi}(t) + \langle \alpha^{\varepsilon}(t), \Tilde{\pi}^{\varepsilon}(t) \rangle \cdot \Tilde{\pi}^{\varepsilon}(t), \qquad t \in (t_k, t_{k+1}),
\end{equation} 
\end{small}
\vspace{-0.5cm}
\begin{small}
\begin{equation}
    \label{eq:filtering_piM_jump_matrix}
    \Tilde{\pi}^{\varepsilon}(t_k) = \frac{\mathcal{B}^{\varepsilon}(t_k^-) \Tilde{\pi}^{\varepsilon}(t_k^-)}{\norm{\mathcal{B}^{\varepsilon}(t_k^-) \Tilde{\pi}^{\varepsilon}(t_k^-)}_1}.
\end{equation}
\end{small}

We show the statement of Theorem~\ref{th:sensitivity} by induction on jump times $t_0, t_1, \dots, t_N$. For $t_0 = 0$, the error is zero due to the equality of the initial distributions. Assume that $\Emu{\norm{\Tilde{\pi}(t) - \Tilde{\pi}^{\varepsilon}(t)}_1} = O(\varepsilon)$ for all $t \in [0, t_k]$ with arbitrary fixed $k$. The goal is to show $\Emu{\norm{\Tilde{\pi}(t) - \Tilde{\pi}^{\varepsilon}(t)}_1} = O(\varepsilon)$ for all $t \in [0, t_{k+1}]$. For $t \in (t_k, t_{k+1})$ from \eqref{eq:filtering_pi_matrix}, \eqref{eq:filtering_piM_matrix}, and the triangle inequality, we obtain
\begin{small}
\begin{equation}
\label{eq:FMP_error_decomp}
\begin{aligned}
    \Emu{\norm{\Tilde{\pi}(t) - \Tilde{\pi}^{\varepsilon}(t)}_1} &\leq \Emu{ \norm{ \Tilde{\pi}(t_k) - \Tilde{\pi}^{\varepsilon}(t_k) }_1 } \\
    &\phantom{\leq} + \Emu{ \int_{t_k}^{t} \norm{\mathcal{A}(s) \Tilde{\pi}(s) - \mathcal{A}^{\varepsilon}(s) \Tilde{\pi}^{\varepsilon}(s) }_1 \diff s } \\ 
    &\phantom{\leq} + \Emu{ \int_{t_k}^{t} \norm{\langle \alpha(s), \Tilde{\pi}(s) \rangle \cdot \Tilde{\pi}(s) - \langle \alpha^{\varepsilon}(s), \Tilde{\pi}^{\varepsilon}(s) \rangle \cdot \Tilde{\pi}^{\varepsilon}(s)}_1 \diff s }.
\end{aligned}
\end{equation}
\end{small}

Next, we apply the triangle inequality to the second term:
\begin{small}
\begin{align*}
    &\Emu{ \int_{t_k}^{t} \norm{\mathcal{A}(s) \Tilde{\pi}(s) - \mathcal{A}^{\varepsilon}(s) \Tilde{\pi}^{\varepsilon}(s) }_1 \diff s } \\
    &\leq \Emu{ \int_{t_k}^{t} \norm{\mathcal{A}(s) \Tilde{\pi}(s) - \mathcal{A}(s) \Tilde{\pi}^{\varepsilon}(s) }_1 \diff s } + \Emu{ \int_{t_k}^{t} \norm{\mathcal{A}(s) \Tilde{\pi}^{\varepsilon}(s) - \mathcal{A}^{\varepsilon}(s) \Tilde{\pi}^{\varepsilon}(s) }_1 \diff s } \\
    &\leq \Emu{ \int_{t_k}^{t} \norm{ \mathcal{A}(s) }_1 \norm{ \Tilde{\pi}(s) -  \Tilde{\pi}^{\varepsilon}(s) }_1 \diff s } + \Emu{ \int_{t_k}^{t} \norm{ \mathcal{A}(s) - \mathcal{A}^{\varepsilon}(s) }_1 \norm{ \Tilde{\pi}^{\varepsilon}(s) }_1 \diff s }.
\end{align*}
\end{small}
The assumption \eqref{eq:assumption_bound_propens_FMP} implies $\norm{\mathcal{A}(s)}_1 \leq 2 C_2(\boldsymbol{y}(t_k))$. In addition, $\norm{\Tilde{\pi}^{\varepsilon}(s)}_1 = 1$ and, according to \eqref{eq:a_tilde_error_eps}, $\Emu{\norm{\mathcal{A}(s) - \mathcal{A}^{\varepsilon}(s)}_1} = O(\varepsilon)$ for all $s \in (t_k, t)$. Therefore, the second term in $\eqref{eq:FMP_error_decomp}$ is bounded by the following:
\begin{small}
\begin{align*}
    2 C_2(\boldsymbol{y}(t_k)) \cdot \Emu{ \int_{t_k}^{t} \norm{ \Tilde{\pi}(s) -  \Tilde{\pi}^{\varepsilon}(s) }_1 \diff s } + C \varepsilon \cdot (t-t_k).
\end{align*}
\end{small}

We also use the triangle inequality to bound the third term in \eqref{eq:FMP_error_decomp}:
\begin{small}
\begin{align*}
    &\Emu{ \int_{t_k}^{t} \norm{\langle \alpha(s), \Tilde{\pi}(s) \rangle \cdot \Tilde{\pi}(s) - \langle \alpha^{\varepsilon}(s), \Tilde{\pi}^{\varepsilon}(s) \rangle \cdot \Tilde{\pi}^{\varepsilon}(s)}_1 \diff s } \\
    %&\leq \Emu{ \int_{t_k}^{t} \norm{\langle \alpha(s), \Tilde{\pi}(s) \rangle \cdot \Tilde{\pi}(s) - \langle \alpha(s), \Tilde{\pi}(s) \rangle \cdot \Tilde{\pi}^{\varepsilon}(s)}_1 \diff s } \\
    %&\phantom{\leq} + \Emu{ \int_{t_k}^{t} \norm{\langle \alpha(s), \Tilde{\pi}(s) \rangle \cdot \Tilde{\pi}^{\varepsilon}(s) - \langle \alpha^{\varepsilon}(s), \Tilde{\pi}^{\varepsilon}(s) \rangle \cdot \Tilde{\pi}^{\varepsilon}(s)}_1 \diff s } \\
    &\leq \Emu{ \int_{t_k}^{t} \abs{\langle \alpha(s), \Tilde{\pi}(s) \rangle}  \norm{\Tilde{\pi}(s) - \Tilde{\pi}^{\varepsilon}(s)}_1 \diff s } \\
    &\phantom{\leq} + \Emu{ \int_{t_k}^{t} \norm{\Tilde{\pi}^{\varepsilon}(s)}_1 \abs{\langle \alpha(s), \Tilde{\pi}(s) \rangle - \langle \alpha^{\varepsilon}(s), \Tilde{\pi}^{\varepsilon}(s) \rangle } \diff s } \\
    &\leq \Emu{ \int_{t_k}^{t} \norm{\alpha(s)}_{\infty} \norm{\Tilde{\pi}(s) - \Tilde{\pi}^{\varepsilon}(s)}_1 \diff s } + \Emu{ \int_{t_k}^{t} \abs{\langle \alpha(s), \Tilde{\pi}(s) \rangle - \langle \alpha^{\varepsilon}(s), \Tilde{\pi}^{\varepsilon}(s) \rangle } \diff s }.
\end{align*}
\end{small}
The last inequality follows from H\"{o}lder's inequality, and $\norm{\Tilde{\pi}(s)} = \norm{\Tilde{\pi}^{\varepsilon}(s)} = 1$. To further bound this expression, we use \eqref{eq:assumption_bound_propens_FMP} to obtain $\norm{\alpha(s)}_{\infty} = C_2(\boldsymbol{y}(t_k))$ and apply the triangle inequality again:
\begin{small}
\begin{align*}
    &C_2(\boldsymbol{y}(t_k)) \cdot \Emu{ \int_{t_k}^{t} \norm{\Tilde{\pi}(s) - \Tilde{\pi}^{\varepsilon}(s)}_1 \diff s } \\
    &\phantom{\leq} + \Emu{ \int_{t_k}^{t} \abs{\langle \alpha(s), \Tilde{\pi}(s) - \Tilde{\pi}^{\varepsilon}(s) \rangle } \diff s } 
    + \Emu{ \int_{t_k}^{t} \abs{\langle \alpha(s) - \alpha^{\varepsilon}(s), \Tilde{\pi}^{\varepsilon} (s) \rangle } \diff s } \\
    &\leq C_2(\boldsymbol{y}(t_k)) \cdot \Emu{ \int_{t_k}^{t} \norm{\Tilde{\pi}(s) - \Tilde{\pi}^{\varepsilon}(s)}_1 \diff s } \\
    &\phantom{\leq} + \Emu{ \int_{t_k}^{t}  \norm{\alpha(s)}_{\infty} \cdot \norm{ \Tilde{\pi}(s) - \Tilde{\pi}^{\varepsilon}(s) }_1 \diff s } 
    + \Emu{ \int_{t_k}^{t} \norm{ \alpha(s) - \alpha^{\varepsilon}(s) }_{\infty} \cdot \norm{ \Tilde{\pi}^{\varepsilon} (s) }_1 \diff s } \\
    &\leq 2 C_2(\boldsymbol{y}(t_k)) \cdot \Emu{ \int_{t_k}^{t} \norm{\Tilde{\pi}(s) - \Tilde{\pi}^{\varepsilon}(s)}_1 \diff s } + C \varepsilon \cdot (t - t_k).
\end{align*}
\end{small}
The last line follows from \eqref{eq:assumption_bound_propens_FMP} and \eqref{eq:a_tilde_error_eps}.

Gathering all bounds back into \eqref{eq:FMP_error_decomp} yields
\begin{small}
\begin{align*}
    \Emu{\norm{\Tilde{\pi}(t) - \Tilde{\pi}^{\varepsilon}(t)}_1} &\leq \Emu{ \norm{ \Tilde{\pi}(t_k) - \Tilde{\pi}^{\varepsilon}(t_k) }_1 } \\
    &\phantom{\leq} +  4 C_2(\boldsymbol{y}(t_k)) \cdot \Emu{ \int_{t_k}^{t} \norm{\Tilde{\pi}(s) - \Tilde{\pi}^{\varepsilon}(s)}_1 \diff s }
    + 2C \varepsilon \cdot (t - t_k).
\end{align*}
\end{small}
Using Gr\"{o}nwall's inequality for $t \mapsto \Emu{\norm{ \Tilde{\pi}(t) - \Tilde{\pi}^{\varepsilon} (t) }_1}$ results in
\begin{small}
\begin{align*}
    \Emu{\norm{\Tilde{\pi}(t) - \Tilde{\pi}^{\varepsilon}(t)}_1} &\leq \left[ \Emu{ \norm{ \Tilde{\pi}(t_k) - \Tilde{\pi}^{\varepsilon}(t_k) }_1 } + 2C \varepsilon \cdot (t - t_k) \right] \\
    &\phantom{\leq} \times \exp \big( 4 C_2(\boldsymbol{y}(t_k)) \cdot (t-t_k) \big),
\end{align*}
\end{small}
that is, 
\begin{small}
\begin{align*}
    \Emu{\norm{\Tilde{\pi}(t) - \Tilde{\pi}^{\varepsilon}(t)}_1} = O(\varepsilon) \quad \text{for all} \quad t \in (t_k, t_{k+1}). 
\end{align*}
\end{small}

To show the result for $t = t_{k+1}$, \eqref{eq:filtering_pi_jump_matrix} and \eqref{eq:filtering_piM_jump_matrix} are used as follows:
\begin{small}
\begin{equation}
\label{eq:FMP_error_decomp_jump}
\begin{aligned}
    \Emu{\norm{\Tilde{\pi}(t_{k+1}) - \Tilde{\pi}^{\varepsilon}(t_{k+1})}_1} &= \Emu{\norm{ \frac{\mathcal{B}(t_{k+1}^-) \Tilde{\pi}(t_{k+1}^-)}{\norm{\mathcal{B}(t_{k+1}^-) \Tilde{\pi}(t_{k+1}^-)}_1} - \frac{\mathcal{B}^{\varepsilon}(t_{k+1}^-) \Tilde{\pi}^{\varepsilon}(t_{k+1}^-)}{\norm{\mathcal{B}^{\varepsilon}(t_{k+1}^-) \Tilde{\pi}^{\varepsilon}(t_{k+1}^-)}_1} }_1} \\
    &\leq \Emu{\norm{ \frac{\mathcal{B}(t_{k+1}^-) \Tilde{\pi}(t_{k+1}^-)}{\norm{\mathcal{B}(t_{k+1}^-) \Tilde{\pi}(t_{k+1}^-)}_1} - \frac{\mathcal{B}^{\varepsilon}(t_{k+1}^-) \Tilde{\pi}^{\varepsilon}(t_{k+1}^-)}{\norm{\mathcal{B}(t_{k+1}^-) \Tilde{\pi}(t_{k+1}^-)}_1} }_1} \\
    &\phantom{\leq} + \Emu{\norm{ \frac{\mathcal{B}^{\varepsilon}(t_{k+1}^-) \Tilde{\pi}^{\varepsilon}(t_{k+1}^-)}{\norm{\mathcal{B}(t_{k+1}^-) \Tilde{\pi}(t_{k+1}^-)}_1} - \frac{\mathcal{B}^{\varepsilon}(t_{k+1}^-) \Tilde{\pi}^{\varepsilon}(t_{k+1}^-)}{\norm{\mathcal{B}^{\varepsilon}(t_{k+1}^-) \Tilde{\pi}^{\varepsilon}(t_{k+1}^-)}_1} }_1}, \\
\end{aligned}
\end{equation}
\end{small}
where the first term can be bounded by
\begin{small}
\begin{align*}
    \dots %&= \frac{1}{\norm{\mathcal{B}(t_{k+1}^-) \Tilde{\pi}(t_{k+1}^-)}_1} \Emu{\norm{ \mathcal{B}(t_{k+1}^-) \Tilde{\pi}(t_{k+1}^-) - \mathcal{B}^{\varepsilon}(t_{k+1}^-) \Tilde{\pi}^{\varepsilon}(t_{k+1}^-)  }_1 } \\
    &\leq \frac{1}{\norm{\mathcal{B}(t_{k+1}^-) \Tilde{\pi}(t_{k+1}^-)}_1} \Emu{\norm{\mathcal{B}(t_{k+1}^-)}_1 \cdot \norm{ \Tilde{\pi}(t_{k+1}^-) - \Tilde{\pi}^{\varepsilon}(t_{k+1}^-) }_1 } \\
    &\phantom{\leq} + \frac{1}{\norm{\mathcal{B}(t_{k+1}^-) \Tilde{\pi}(t_{k+1}^-)}_1} \Emu{\norm{ \mathcal{B}(t_{k+1}^-) - \mathcal{B}^{\varepsilon}(t_{k+1}^-)}_1 \cdot \norm{\Tilde{\pi}^{\varepsilon}(t_{k+1}^-) }_1 } \\
    &\leq \frac{\norm{\mathcal{B}(t_{k+1}^-)}_1}{\norm{\mathcal{B}(t_{k+1}^-) \Tilde{\pi}(t_{k+1}^-)}_1} \Emu{ \norm{ \Tilde{\pi}(t_{k+1}^-) - \Tilde{\pi}^{\varepsilon}(t_{k+1}^-) }_1 } + \frac{1}{\norm{\mathcal{B}(t_{k+1}^-) \Tilde{\pi}(t_{k+1}^-)}_1} \cdot C \varepsilon,
\end{align*}
\end{small}
and the second term can be bounded by
\begin{small}
\begin{align*}
    \dots &= \Emu{ \norm{ \mathcal{B}^{\varepsilon}(t_{k+1}^-) \Tilde{\pi}^{\varepsilon}(t_{k+1}^-) }_1 \cdot \abs{ \frac{1}{\norm{\mathcal{B}(t_{k+1}^-) \Tilde{\pi}(t_{k+1}^-)}_1} - \frac{1}{\norm{\mathcal{B}^{\varepsilon}(t_{k+1}^-) \Tilde{\pi}^{\varepsilon}(t_{k+1}^-)}_1} } } \\
    &= \Emu{ \norm{ \mathcal{B}^{\varepsilon}(t_{k+1}^-) \Tilde{\pi}^{\varepsilon}(t_{k+1}^-) }_1 \cdot \abs{ \frac{ \norm{\mathcal{B}^{\varepsilon}(t_{k+1}^-) \Tilde{\pi}^{\varepsilon}(t_{k+1}^-)}_1 - \norm{\mathcal{B}(t_{k+1}^-) \Tilde{\pi}(t_{k+1}^-)}_1 }{\norm{\mathcal{B}(t_{k+1}^-) \Tilde{\pi}(t_{k+1}^-)}_1 \cdot \norm{\mathcal{B}^{\varepsilon}(t_{k+1}^-) \Tilde{\pi}^{\varepsilon}(t_{k+1}^-)}_1} } } \\
    &= \frac{1}{\norm{\mathcal{B}(t_{k+1}^-) \Tilde{\pi}(t_{k+1}^-)}_1} \Emu{ \abs{ \norm{\mathcal{B}^{\varepsilon}(t_{k+1}^-) \Tilde{\pi}^{\varepsilon}(t_{k+1}^-)}_1 - \norm{\mathcal{B}(t_{k+1}^-) \Tilde{\pi}(t_{k+1}^-)}_1 } } .
\end{align*}
\end{small}
Both $\mathcal{B}(t_{k+1}^-) \Tilde{\pi}(t_{k+1}^-)$ and $\mathcal{B}^{\varepsilon}(t_{k+1}^-) \Tilde{\pi}^{\varepsilon}(t_{k+1}^-)$ are elementwise nonnegative; thus, the second term is bounded by
\begin{small}
\begin{align*}
    \dots &\leq \frac{1}{\norm{\mathcal{B}(t_{k+1}^-) \Tilde{\pi}(t_{k+1}^-)}_1} \Emu{ \norm{\mathcal{B}^{\varepsilon}(t_{k+1}^-) \Tilde{\pi}^{\varepsilon}(t_{k+1}^-) - \mathcal{B}(t_{k+1}^-) \Tilde{\pi}(t_{k+1}^-)}_1 }.
\end{align*}
\end{small}
This expression is exactly the first term; therefore, it has the same upper bound.

Inserting all bounds back into \eqref{eq:FMP_error_decomp_jump} yields
\begin{small}
\begin{align*}
    \Emu{\norm{\Tilde{\pi}(t_{k+1}) - \Tilde{\pi}^{\varepsilon}(t_{k+1})}_1} 
    &\leq \frac{2 \norm{\mathcal{B}(t_{k+1}^-)}_1}{\norm{\mathcal{B}(t_{k+1}^-) \Tilde{\pi}(t_{k+1}^-)}_1} \Emu{ \norm{ \Tilde{\pi}(t_{k+1}^-) - \Tilde{\pi}^{\varepsilon}(t_{k+1}^-) }_1 } \\
    &\phantom{\leq} + \frac{2}{\norm{\mathcal{B}(t_{k+1}^-) \Tilde{\pi}(t_{k+1}^-)}_1} \cdot C \varepsilon.
\end{align*}
\end{small}
Previously, we have shown that $\Emu{ \norm{ \Tilde{\pi}(t_{k+1}^-) - \Tilde{\pi}^{\varepsilon}(t_{k+1}^-) }_1 } = O(\varepsilon)$; therefore, $$\Emu{ \norm{ \Tilde{\pi}(t_{k+1}) - \Tilde{\pi}^{\varepsilon}(t_{k+1}) }_1 } = O(\varepsilon).$$ 

We have shown that, under the induction assumption, $\Emu{ \norm{ \Tilde{\pi}(t) - \Tilde{\pi}^{\varepsilon}(t) }_1 } = O(\varepsilon)$ for all $t \in [0, t_{k+1}]$. We conclude that the statement of Theorem~\ref{th:sensitivity} is true for all $t \in [0, t_N] = [0, T]$.

\end{proof}
    
    \bibliographystyle{plain}
    \bibliography{main} 
        
\end{document}